\documentclass[11pt]{amsart}
\usepackage{amscd,amssymb,graphics}
\usepackage[hidelinks]{hyperref}
\usepackage{amsfonts}
\usepackage{amsmath}
\usepackage{enumerate}

\input xy
\xyoption{all}

\numberwithin{equation}{section}

\newtheorem{thm}{Theorem}[section]

\newtheorem{corol}[thm]{Corollary}
\newtheorem{lemma}[thm]{Lemma}
\newtheorem{prop}[thm]{Proposition}

\newtheorem{quest}[thm]{Question}
\theoremstyle{definition}
\newtheorem{defi}[thm]{Definition}

\newtheorem{notation}[thm]{Notation}

\theoremstyle{remark}
\newtheorem{remark}[thm]{Remark}
\newtheorem{example}[thm]{Example}

\begin{document}
	\title[]{Non-archimedean transportation problems and Kantorovich ultra-norms}

	\author[]{Michael Megrelishvili}
	\address{Department of Mathematics,
		Bar-Ilan University, 52900 Ramat-Gan, Israel}
	\email{megereli@math.biu.ac.il}
	\urladdr{http://www.math.biu.ac.il/$^\sim$megereli}

	\author[]{Menachem Shlossberg}
	\address{Department of Mathematics,
		Bar-Ilan University, 52900 Ramat-Gan, Israel}
		\email{menysh@yahoo.com}
	\urladdr{http://www.shlossberg.com}
	
	\date{April 12, 2016}

		\begin{abstract}
			We study a non-archimedean (NA) version of transportation problems and introduce naturally arising ultra-norms which we call \emph{Kantorovich ultra-norms}.
			For every ultra-metric space and every  
			 NA valued field
			 (e.g., the field $\mathbb Q_{p}$ of $p$-adic numbers)
			the naturally defined inf-max cost formula achieves its infimum.
			We also present NA versions of the Arens-Eells construction and of the integer value property.
			We introduce and study \emph{free NA locally convex spaces}. In particular, we provide conditions under which these spaces are normable by Kantorovich ultra-norms and also conditions which yield NA versions of Tkachenko-Uspenskij theorem about free abelian topological groups.
		\end{abstract}
		
	

	\thanks{This research was supported by a grant of Israel Science Foundation (ISF 668/13)}
	
		\maketitle
	\setcounter{tocdepth}{1}
	
	\tableofcontents

	\section{Introduction}\label{sec:introduction}
	
	Kantorovich norm \cite{kan} plays a major role in various areas of mathematics, economics and computer science (see \cite{DengDu,GaoPest,MPV,pest-wh,RachRush,Sip,Ver,Villani}),
	for instance, in  Monge-Kantorovich transportation problem.
	  The 
	  seminorms that determine the topology of the free real locally convex space are in fact Kantorovich
	  seminorms (see \cite{GaoPest,MPV,pest-wh,Sip}). Uspenskij \cite{Us-free}  provided a simplified formula for these 
	  seminorms.
	
	 In this paper we deal with discrete transportation problems.
	 In Subsection \ref{sub:dem} we present a slightly more flexible ("democratic")  approach to the classical  Kantorovich problem. This approach is related to the \emph{transshipment problem}. Continuing in this direction, in Section \ref{s:naTP}  we study  \emph{non-archimedean transportation problems}. Since the term \emph{non-archimedean} appears many times in this work,  we write shortly: NA.
	
	 In Section \ref{sec:kanult} we present an NA version of the Arens-Eells embedding (Theorem \ref{t:AE}).
	 We introduce the naturally arising Kantorovich ultra-(semi)norms, defined for an ultra-(pseudo)metric space $(X,d)$ on free vector spaces $L_\mathbb{F}(X)$ via min-max formula.    Theorem \ref{t:AE} shows  that for an arbitrary NA valued field $\mathbb F$ and for any 
	 $u\in L_{\mathbb F}(X)$ the value of the Kantorovich ultra-(semi)norm $||u||$ can be approximated as
	$$||u||=\inf\bigg \{\max \limits_{1\leq i\leq k}| s_i|d(x_i,y_i): u=\sum\limits_{i=1}^{k} s_i(x_i-y_i), \ x_i,y_i\in \operatorname{supp}(u), \ s_i\in \mathbb F \bigg \},$$ 
	 where $\operatorname{supp}(u)$ is the support of $u$
	 (see Notation \ref{notation}).

	  Note that the analogous property in the archimedean case does not hold in general. Indeed,  it is no longer true when  $\mathbb F=\Bbb C$, the field of complex numbers,
	  in contrast to the case $\mathbb F=\Bbb R$
	  (see \cite{Flood,Wea} and Remark \ref{r:3}.3 below).

	The infimum in Theorem \ref{t:AE} is, in fact, a minimum. This refinement, which comes from
	  Min-attaining Theorem \ref{t:attaining},
	  provides another contrast to the archimedean case. Indeed, in the Appendix we give an  example in which  the infimum is not attained for $\mathbb F=\Bbb Q(i).$
	 Another  refinement concerns the coefficients (the \emph{$G$-value property}), that is,
	 it is enough to take the coefficients from the additive subgroup $G_u$ of $\mathbb F$ generated by the normal coefficients $\lambda_i$ of
	 $u=\sum\limits_{i=1}^n \lambda_i x_i.$
	  Namely, we show that
	 $$||u||=\min \bigg \{\max \limits_{1\leq i,j \leq m}|c_{ij}|d(x_i,x_j): c_{ij} \in \mathbb F,   \ \forall i:  1\leq i\leq n \ \sum\limits_{j=1}^{m}c_{ij}- \sum\limits_{j=1}^{m}c_{ji}=\lambda_i  \bigg\},$$
	 such that all coefficients $c_{ij}$ belong to $G_u$. Note that
	 a matrix $(c_{ij})\in \mathbb F^{m\times m}$ satisfies the  equations
	  $\sum\limits_{j=1}^{m}c_{ij}- \sum\limits_{j=1}^{m}c_{ji}=\lambda_i \ \forall i: 1\leq i\leq n$ if and only if $u=\sum\limits_{i=1}^{m}\sum\limits_{j=1}^{m}c_{ij}(x_i-x_j).$
	 As a 
	 particular case we get an NA generalization of the so-called \emph{integer value property}
	 (well known in case $\mathbb F=\Bbb R$).

	Probably one can encounter a variety of min-max optimization problems when dealing with the Kantorovich ultra-norms.
	It is worth noting that different algorithms for solving such problems are known (see \cite{BDM} for example).
	
	 In Section \ref{s:freeLCS} we introduce the \emph{free NA locally convex spaces} for NA uniform spaces. We describe their topologies in terms of Kantorovich ultra-seminorms (Theorem \ref{t:freeLCS}). We show that
	 for an ultra-metric space $(X,d)$ and a trivially valued field $\mathbb F$, the free NA
	 locally convex space $ L_\mathbb{F}(X,\mathcal{U}(d))$ (of the uniformity $\mathcal{U}(d)$ of $d$) is normable by the Kantorovich ultra-norm induced by $d$  (Theorem \ref{t:normable}).
	  By Tkachenko-Uspenskij theorem (in the
	  archimedean case $\mathbb F=\Bbb R$) the free abelian topological group $A(X)$ is a topological subgroup of $L(X)$. Using Ostrowski's classical theorem we prove that in case $\mathbb F$ is  an NA valued field
	  of zero characteristic, the uniform free NA   abelian topological group
	 $A_{\scriptscriptstyle\mathcal{NA}}(X,\mathcal{U})$ is a topological subgroup of  $ L_\mathbb{F}(X,\mathcal U)$ if and only if the restricted valuation on $\Bbb Q$ is trivial (Theorem \ref{t: Tk-Usp}).
For example, this is the case for the Levi-Civita field (Example \ref{Levi-Civita}).

\section{Kantorovich norm}\label{sec:cla}

  For a nonempty set $X$ and a field $\mathbb F$ denote by $L_\mathbb F(X)$ the free  
  $\mathbb F$-vector space on the set $X.$
  We simply write $L(X)$ in case $\mathbb F=\Bbb R$.
  Define
 $\overline{X}:=X\cup \{\textbf{0}\}$ where $\mathbf{0}\notin X$ is the
 zero element of $L_\mathbb F(X).$ The zero element of the field $\mathbb F$ is denoted by $0_\mathbb F.$
 Denote by $L^0_\mathbb F(X)$ the kernel
 of the linear functional $$L_\mathbb F(X)\to \mathbb F, \ \sum\limits_{i=1}^{n}\lambda_i x_i\mapsto \sum\limits_{i=1}^{n}\lambda_i.$$

 	\begin{notation} \label{notation}
 		Every non-zero vector $u \in L_\mathbb{F}(X)$ has a \emph{normal form}
 	  as follows: $u= \sum\limits_{i=1}^{n}\lambda_i x_i\in L_\mathbb{F}(X)$, where
 	 $x_i\in X, \ \lambda_i\in \mathbb F\setminus \{0_{\mathbb F}\}$ $\forall i:  1\leq i\leq n$ and $x_i\neq x_j$   whenever $i\neq j.$ If $u\in L^0_\mathbb F(X)$ then define the {\it support of $u$} as
 	  $\operatorname{supp}(u):=\{x_1, \ldots, x_n\}.$ Otherwise, let $\operatorname{supp}(u):=\{x_1, \ldots, x_n,x_{n+1}\} $ where $x_{n+1}=\mathbf{0}.$
 	 We denote by $m:=|\operatorname{supp}(u)|$ the length of the support, so $m$ is either $n$ or $n+1$.
 	 The support of $\mathbf{0}$ is $\{\mathbf{0}\}$. In what follows, by writing  $u= \sum\limits_{i=1}^{n}\lambda_i x_i\in L_\mathbb{F}(X)$ we mean that it is a normal form.
 	\end{notation}

\subsection{Classical transportation problem}
\label{ss:classical}
  Recall the following transportation problem from the historical work of Kantorovich \cite{kan}.
  Let $(X,d)$ be a metric space and
 denote by  $\mathbb R_{\geq 0}$  the set of non-negative reals.
 Suppose that a network of railways connects a number of production locations  $x_1,\ldots, x_n\in X$  with daily output of $\lambda_1,\ldots, \lambda_n$ carriages  of certain goods, respectively, to a number of consumption locations   $y_1,\ldots, y_m\in X$ with daily demand of $\mu_1,\ldots, \mu_m$  carriages. So, we have $\sum\limits_{i=1}^n\lambda_i=\sum\limits_{j=1}^m\mu_j,$ where $\lambda_i, \mu_j$ are positive.
Let $c_{ij}$ denote the real number transferred from  point $x_i$ to point $y_j.$ We view the metric $d$ as a cost function, and we want to minimize our total sum-cost.
The value we are seeking is
\begin{equation} \label{classform}
\inf\bigg \{\sum\limits_{i=1}^{n}\sum\limits_{j=1}^{m}c_{ij}d(x_i,y_j): c_{ij}\in \mathbb R_{\geq 0}, \sum\limits_{i=1}^{n}c_{ij}=\mu_j, \sum\limits_{j=1}^{m}c_{ij}=\lambda_i\bigg \}.
\end{equation} This infimum is known as the \emph{Kantorovich distance}
in $L(X) $ between $\sum\lambda_ix_i$ and $\sum\mu_jy_j.$
 It coincides  with $||u||$ where $u=\sum\lambda_ix_i-\sum\mu_jy_j\in L^0(X) $ and $||\cdot||$ is the norm defined on $L^0(X)$ as follows. For every $v=\sum\limits_{i=1}^n \lambda_ix_i \in L^0(X) $

  \begin{equation} \label{secform}
||v||=\inf\bigg\{\sum\limits_{i=1}^{l}|\rho_i|d(a_i,b_i):v=\sum\limits_{i=1}^{l}\rho_i(a_i-b_i), \ \rho_i\in \mathbb R, a_i,b_i\in X\bigg\}.
\end{equation}
This norm on $L^0(X)$ is called  the \emph{Kantorovich norm}, \cite{MPV}. If $(X,d)$ is a pseudometric space then (\ref{classform}) and (\ref{secform}) define the \emph{Kantorovich pseudometric} and the  \emph{Kantorovich seminorm} respectively.

Let $X$ be a Tychonoff space. Denote by $D$ the family of all continuous pseudometrics on $\overline{X}:=X\cup \{\textbf{0}\}$. For each $d \in D$ there exists a maximal seminorm $p_d$ on $L(X)$ which extends $d$. We retain the name  \emph{Kantorovich seminorm} for $p_d$ (and for its restriction on $L^0(X)$), although several authors use the name \emph{Kantorovich-Rubinstein seminorm}.
 The vector $\Bbb R$-space $L(X)$ and the family of seminorms $\{p_d: d \in D\}$ determine the free locally convex space over $X$.
See Pestov \cite{pest-wh}, for example, and compare with Raikov \cite{MPV} in the case of \emph{pointed uniform spaces}. In Section \ref{s:freeLCS} we study 
the free NA  locally convex $\mathbb{F}$-space $L_\mathbb{F}(X)$  of an NA uniform space $(X,{\mathcal U}).$ \vskip 0.3cm

Equation (\ref{secform}) has a natural generalization. Let $(\mathbb F, | \cdot|)$ be an archimedean  valued field and $(X,d)$ be a pseudometric space. For every $v=\sum\limits_{i=1}^n \lambda_ix_i \in L_{\mathbb F}^0(X)$
define the \emph{Kantorovich seminorm} as follows:
\begin{equation} \label{arcgen}
||v||=\inf\bigg\{\sum\limits_{i=1}^{l}|\rho_i|d(a_i,b_i):v=\sum\limits_{i=1}^{l}\rho_i(a_i-b_i), \ \rho_i\in \mathbb F, a_i,b_i\in X\bigg\}.
\end{equation}
Note  that every archimedean  valued field $(\mathbb F, | \cdot|)$ is essentially
a subfield of $\Bbb C$ and the valuation is equivalent to the usual valuation on  $\Bbb C$ (see \cite[p. 4]{ro} for example).
 \subsection{"Democratic" reformulation}\label{sub:dem}

We wish to highlight a point that will become important in the sequel. In the problem described above two disjoint  sets $A=\{x_1,\ldots, x_n\} $ and $B=\{y_1,\ldots, y_m\}$ are considered. The distances between the elements in  each set seem irrelevant. Indeed, every distance which appears in Formula (\ref{classform}) is  between an element of  $A$ and an element of $B$.

  Now we consider a more flexible form of the transportation problem (see also  \cite[p. 44]{Wea}).
    Let  $\lambda_1,\ldots, \lambda_n\in \mathbb R$ with $\sum\limits_{i=1}^{n}\lambda_i=0.$ We have to transfer real numbers between the points $x_1,\ldots, x_n\in X$ in the following way.  The sum of  numbers transferred from  $x_i$ minus the sum of  numbers transferred to  $x_i$ is $\lambda_i.$
  Let $c_{ij}$
  denote the real number transferred from  point $x_i$ to point $x_j.$
 We want to minimize  our cost, that is, the  value  of $\sum\limits_{i=1}^{n}\sum\limits_{j=1}^{n}|c_{ij}|d(x_i,x_j)$. 
 Clearly, one may assume that  $c_{ii}=0.$

 \vskip 0.3cm

  As the following lemma suggests, the Kantorovich norm serves both of the   approaches
described above.
  \begin{lemma}[Democratic reformulation] \label{lem:dem}
  	If $v=\sum\limits_{i=1}^n \lambda_ix_i\in L^0 (X),$
 	then
 	\begin{equation}  \label{lastfor} ||v||=\inf\bigg \{\sum\limits_{i=1}^{n}\sum\limits_{j=1}^{n}|c_{ij}|d(x_i,x_j):\sum\limits_{j=1}^{n}c_{ij}- \sum\limits_{j=1}^{n}c_{ji}=\lambda_i \ \forall i \bigg\}.\end{equation}
 	\end{lemma}
 \begin{proof}  Denote by $||v||'$ the expression on the right hand side of
 Equation	(\ref{lastfor}).
 	 We want to show that $||v||=||v||'.$  Let $(c_{ij})\in \mathbb R^{n\times n}$  such that  $\sum\limits_{j=1}^{n}c_{ij}- \sum\limits_{j=1}^{n}c_{ji}=\lambda_i. $ The coefficient of $x_i$ in $\sum\limits_{i=1}^{n}\sum\limits_{j=1}^{n}c_{ij}(x_i-x_j)$ is just $\sum\limits_{j=1}^{n}c_{ij}- \sum\limits_{j=1}^{n}c_{ji}.$
 	It follows that
 	 $$v=\sum\limits_{i=1}^{n}\sum\limits_{j=1}^{n}c_{ij}(x_i-x_j).$$ So, by 
 	 Equation (\ref{secform}) $$||v||\leq  ||v||'.$$ On the other hand, using  reductions
 	from \cite{Us-free}, we show that if $v=\sum\limits_{i=1}^{l}\rho_i(a_i-b_i)$ then there exists a decomposition $v=\sum\limits_{i=1}^{n}\sum\limits_{j=1}^{n}c_{ij}(x_i-x_j)$ with  $$\sum\limits_{i=1}^{n}\sum\limits_{j=1}^{n}|c_{ij}|d(x_i,x_j)\leq \sum\limits_{i=1}^{l}|\rho_i|d(a_i,b_i).$$ To see this, first observe  that we may assume that $\rho_i>0 \ \forall i.$ Consider the following reductions which do not increase the value of the corresponding
	sum:
 	\begin{enumerate}
 	\item Delete any term of $v$ of the form $\rho_i(x-x).$
 	\item If there exist two terms
 	 $\lambda(x_i-x_j)$ and $\mu(x_i-x_j)$ with $\lambda, \mu>0$   replace them with the single  term $(\lambda+\mu)(x_i-x_j).$
 	\item Assuming the decomposition  contains the term $\lambda(x-z)$ where $z\notin \operatorname{supp}(v),$ then we necessarily have
 	also a term of the form $\mu(z-y)$ where $\lambda,\mu>0.$ We have three subcases to consider
 	replacing in each case the terms $\lambda(x-z)$ and $\mu(z-y)$.
\vskip 0.2cm  	
 	\begin{enumerate} [(a)] \item
 	 If $\lambda= \mu$ then  replace
 	 the pair of terms above with one term $\lambda(x-y).$ It is possible since
 	 $\lambda d(x,y)\leq \lambda d(x,z)+\lambda d(z,y).$   \item
 	 If $\lambda<\mu$ then  replace the terms with $\lambda(x-y)$ and $(\mu-\lambda)(z-y).$ The value of the sum does not increase since
 	  \begin{align*}
 	  \lambda d(x,z)+\mu d(z,y) &=\lambda(d(x,z)+d(z,y))+(\mu-\lambda)d(z,y)\geq \\  &\geq \lambda d(x,y)+(\mu-\lambda)d(z,y).
 	  \end{align*}
 	
 	 \item If $\lambda>\mu$ then replace the terms
 	 with $(\lambda-\mu)(x-z)$ and $\mu(x-y).$
 	
 	 This time we have
 	 \begin{align*}
 	 \lambda d(x,z)+\mu d(z,y) &=(\lambda-\mu)d(x,z)+\mu(d(x,z)+d(z,y))\geq \\ &\geq (\lambda-\mu)d(x,z)+\mu d(x,y).
 	 \end{align*}
 	 \end{enumerate}
 	 \end{enumerate}
 	 Using reduction $(3)$  the number of terms containing $z$  decreases. Applying finitely many substitutions of this form and taking into account that the sum of  $z's$ coefficients in any decomposition of $v$ is equal to zero,
 	 we obtain a decomposition of $v$ with only two terms containing $z$: $\lambda(x-z)$ and $\lambda(z-y).$ Now use reduction $(3.a).$ Therefore, we can assume that the decomposition only contains  terms with support elements. That is, terms of the form $\lambda(x_i-x_j) $ where
 	$\lambda\geq 0.$ At this point we use reduction $(2)$ if necessary.  We obtain a decomposition $v=\sum\limits_{i=1}^{n}\sum\limits_{j=1}^{n}c_{ij}(x_i-x_j)$ with  $$\sum\limits_{i=1}^{n}\sum\limits_{j=1}^{n}|c_{ij}|d(x_i,x_j)\leq \sum\limits_{i=1}^{l}|\rho_i|d(a_i,b_i).$$
 	It follows that $||v||'\leq  ||v||$
 	and  we conclude that $||v||=||v||'.$
 	\end{proof}

 \begin{remark} \label{r:3} \
 	\begin{enumerate}
 	\item
Every non-zero element $v\in L^0(X)$  has the form
 $v=\sum\limits_{i=1}^{n}	a_ix_i-\sum\limits_{j=1}^{m}	b_jy_j $ where
 $ \ \sum\limits_{i=1}^n a_i=\sum\limits_{j=1}^m b_j$ and
  $\forall i:  1\leq i\leq n \ \forall j:  1\leq j \leq m \ a_i,b_j>0.$
  Using this fact  one can  move back from the democratic approach to the
  classical one as in Section \ref{ss:classical}.
  \item  Using compactness arguments one can prove that the infimum  in Formula (\ref{classform}) is attained. By the proof of  Lemma \ref{lem:dem}, for any minimizing matrix $(c_{ij})$ from (\ref{classform}) there exists a matrix $(t_{ij})$ from  (\ref{lastfor}) such that $$\sum\limits_{i=1}^{n}\sum\limits_{j=1}^{n}|t_{ij}|d(x_i,x_j)\leq \sum\limits_{i=1}^{n}\sum\limits_{j=1}^{n}|c_{ij}|d(x_i,x_j).$$ It follows that the infimum in (\ref{lastfor}) is attained at $(t_{ij})$.
  \item  Replacing $\Bbb R$ with  $\Bbb C$  completely changes the situation.
  As it follows from
  \cite{Flood,Wea}, in the latter case of $\mathbb F=\Bbb C$ we cannot even guarantee that the infimum in 
 (\ref{secform}) can be approximated by computations on support elements of a vector $u \in L_{\Bbb C}^0(X)$.
 A detailed example is provided in the Appendix (Theorem \ref{main} and Example \ref{e:outside}).
  \end{enumerate}
 \end{remark}

\section{Non-archimedean transportation problem}
\label{s:naTP}

In this section we  discuss the main object of our work: a {\it non-archimedean transportation problem} (NATP).
First we recall some definitions.

\subsection{Preliminaries}
A metric space $(X,d)$ is  an \emph{ultra-metric space}
if $d$ is an \emph{ultra-metric}, i.e., it satisfies {\it the strong
	triangle inequality}
$$d(x,z) \leq \max \{d(x,y), d(y,z)\}.$$
Allowing the distance between distinct elements to be zero we obtain
the definition of  an \emph{ultra-pseudometric}.
It is well known that if $d(x,y) \neq d(y,z)$ then $d(x,z) = \max \{d(x,y), d(y,z)\}.$

A uniform space $(X,{\mathcal U})$ is NA if it has a base $B$ consisting of equivalence relations on $X.$ For every ultra-pseudometric $d$ on $X$
the open balls of radius $\varepsilon >0$ form a clopen partition of
$X.$ So, the uniformity induced by any ultra-pseudometric $d$ on $X$
is NA. A uniformity is NA if and only if
it is generated by a system $\{d_i\}_{i \in I}$ of {\it ultra-pseudometrics}.

Recall that a topological group is {\it non-archimedean} if it has a base at the identity consisting of open subgroups. For some properties of this class of topological groups see for example \cite{MS1,MS}. We say that a topological ring (or field or vector space) is NA if its
additive group is NA. Note that Lyudkovskii \cite{Lyu} studied NA free Banach spaces.

A {\it valuation} on  a field $\mathbb F$ is a function $|\cdot|: \mathbb F\to [0,\infty)$ which satisfies the following  ($x,y\in \mathbb F$):
\begin{enumerate}
 \item $|x|\geq 0$;
\item $|x|=0$ if and only if $x=0_{\mathbb F}$;
\item $|x+y|\leq |x|+|y|$;
\item $|xy|=|x||y|$.
\end{enumerate}

  Replacing condition $(3)$ with
  $|x+y|\leq \max\{|x|,|y|\}$ we obtain a {\it non-archimedean valuation}. In this case the metric $d$ defined by $d(x,y)=|x-y|$
is an ultra-metric.

 An (NA) {\it valued field} is a field $\mathbb F$ with a (resp., NA) valuation $|\cdot|$. Every NA valued field is NA as a topological group because every open ball $\{x \in \mathbb F: |x| <r\}$ is a (clopen) additive subgroup.

A valuation which is not NA is called  an {\it archimedean valuation}.
Let $(\mathbb F, | \cdot|)$ be a  valued field.   A \emph{seminorm} on an $\mathbb F$-vector space $V$ is
a map  $||\cdot||: V \to [0,\infty)$ such that ($x,y\in V, \alpha\in\mathbb F$):
\begin{enumerate}
\item $||0_{V}||=0$;
\item $||x+y|| \leq ||x||+||y||;$
\item $||\alpha x||=|\alpha| ||x||.$
\end{enumerate}
If instead of condition $(1)$ we have: $||x||=0$ if and only if $x=0_{V}$,
then  $||\cdot||$ is called a {\it norm}. If the valuation on $\mathbb F$ is NA and condition $(2)$ is replaced by $||x+y|| \leq \max \{||x||,||y||\}$,    then the  norm (seminorm) $||\cdot||$ is an {\it ultra-norm} (respectively, ultra-seminorm).

Let $(\mathbb F,|\cdot|)$ be an NA valued field. The set $\{|x|: |x|\neq 0\}$ is a subgroup of the multiplicative group $\Bbb R_{>0}$ of all positive reals and is said to be the {\it value group} of the
valuation $|\cdot|.$ The value group is either discrete or dense in $\Bbb R_{>0}$.
Accordingly the valuation is called {\it discrete} or {\it dense}.
If the value group is the trivial subgroup $\{1\}$ then the valuation is said to be {\it trivial}. For any non-trivial discrete valuation the value group is the infinite cyclic closed subgroup $\{a^k: k \in \Bbb Z\}$ of $\Bbb R_{>0}$, where $a:=\max\{|x|: |x| <1 \}$.

Note that  discretely valued fields form a  major subclass in the class of NA valued fields.
This subclass is closed under taking arbitrary subfields, completions and finite extensions.

The $p$-adic valuation on the field $\Bbb Q$ of rationals
is a classical particular case (for every prime $p$).
The completion is
the field of $p$-adic numbers $\mathbb Q_{p}$, a locally compact NA valued field.  The valuation of every locally compact NA valued field is discrete (see \cite{ro}). The natural valuation
 on the field $\Bbb C\{\{T\}\}$ of formal Laurent series (which is not locally compact) is discrete
 \cite{Schn}.

 Below we use several times the following well known theorem of Ostrowski (see for example \cite[Theorem 1.2]{ro}) which shows that the $p$-adic valuation, up to 
 a natural equivalence, is the only NA non-trivial valuation on $\Bbb Q$. In particular, any NA valuation on $\Bbb Q$ is discrete.

 \begin{thm} \label{Ostrowski}
 	\emph{(Ostrowski's Theorem)} Let $|\cdot|$ be a non-trivial NA valuation on the field $\Bbb Q$ of rationals. Then there exists a prime $p$ such that  $|\cdot|$ is equivalent to the
 	$p$-adic valuation $|\cdot|_p$ (namely, there exists $c>0$ such that $|x|=|x|_p^c \ \ \forall x \in \Bbb Q$).
 \end{thm}

 The following is an important example of a densely valued NA field.

  \begin{example} \label{Levi-Civita1}
  	 Recall that the elements of the Levi-Civita field $\mathcal R$  (see \cite{Sham} for example) are
  	 real functions   	 
  	 $f: \Bbb Q \to \Bbb R$ with left-finite
  	support. That is, for every rational number $q$ the set $A_q:=\{a<q| \ f(a)\neq 0\}$ is finite. The field operations are addition and convolution. $\Bbb R$ is (algebraically) isomorphic to a subfield of $\mathcal R.$  Indeed, the map $a \mapsto f_a$ from $\Bbb R$ to $\mathcal R$, where $f_a(0)=a$ and $f_a(x)=0 \ \forall x\neq 0,$ is a field embedding.
  
   For every non-zero element 	$f\in \mathcal R$,  the support of $f$ (notation: $\operatorname{supp}(f)$) has a minimum,	due to its left-finiteness.
  	Recall that $\mathcal R$ admits a natural NA valuation  defined by  $|f|=e^{-\min \operatorname{supp}(f)}$ for non-zero $f$.
    It is easy to see that this valuation is dense.
  	At the same time the restricted valuation on $\Bbb Q$ is trivial.
  \end{example}

\subsection{Formulation of NATP} \label{s:NATP}

We formulate here a non-archimedean transportation problem using a democratic approach (compare Section \ref{lem:dem}).
Let $\mathbb F$ be an NA valued field, $(X,d)$ be an ultra-(pseudo)metric space and   $x_i\in X$ for every $1\leq i\leq n$.

We have to transfer field elements between these points in the following way.  The sum of elements transferred from  $x_i$ minus the sum of elements transferred to $x_i$ is $\lambda_i,$ where
$\lambda_1,\ldots, \lambda_n$ are given  elements in $\mathbb F$ with 
$\sum\limits_{i=1}^{n}\lambda_i=0_{\mathbb F}.$

Let $c_{ij} \in \mathbb F$
denote the element transferred from  $x_i$ to $x_j.$ 
 Note that by the  setting of NATP we have $\forall  i \ \sum\limits_{j=1}^{n}c_{ij}- \sum\limits_{j=1}^{n}c_{ji}=\lambda_i.$ We want to minimize as much as possible our max-cost, that is, the  value  of $$\max\limits_{1\leq i,j\leq n}|c_{ij}|d(x_i,x_j).$$
	A natural question arises:
 \begin{quest} \label{quest:na} Is the infimum  \begin{equation} \label{mul:inf}
 	\inf\bigg \{\max \limits_{1\leq i, j \leq n}|c_{ij}|d(x_i,x_j): \forall  i \ \sum\limits_{j=1}^{n}c_{ij}- \sum\limits_{j=1}^{n}c_{ji}=\lambda_i  \bigg\}
 	\end{equation}
 	 attained?
\end{quest}

Min-attaining Theorem \ref{t:attaining} implies that the answer to Question \ref{quest:na} is positive
for every  NA valued field $\mathbb F$ (e.g., $\mathbb Q_{p}$)
and any ultra-(pseudo)metric space $(X,d).$

   In fact we will show in Theorem \ref{t:AE} that  (\ref{mul:inf}) can be studied via a special ultra-(semi)norm $||\cdot||_d$ on $L_{\mathbb F}(X)$. We call it the \emph{Kantorovich ultra-(semi)norm} associated with $d$  (Definition \ref{d:KantUltraNorm})
   because its role is similar to the role of the Kantorovich
 (semi)norm in the classical  transportation problem (with $\mathbb F=\Bbb R$).
   Indeed, the infimum in (\ref{mul:inf}) coincides with $||u||_d,$ where
   $u=\sum\limits_{i=1}^{n}\lambda_ix_i \in L^0_\mathbb F(X).$

	\section{Kantorovich ultra-norms} \label{sec:kanult}

Let $(X,d)$ be an ultra-pseudometric space. Consider the set $\overline{X}:=X\cup \{\textbf{0}\},$ where $\textbf{0} \notin X$. In the sequel we repeatedly use the following simple lemma.

\begin{lemma} \label{l:extend}
For every ultra-pseudometric $d$ on $X$ there exists an ultra-pseudometric (denoted also by $d$) which extends $d$ on $\overline{X}:=X\cup \{\mathbf{0}\},$ such that $\mathbf{0}$ is an isolated point
in $(\overline{X},d)$.
\end{lemma}
\begin{proof}
	Fix $x_0\in X$ and extend the definition of $d$ from $X$ to $\overline{X}$ by letting $d(x,\mathbf{0})=\max\{d(x,x_0),1\}.$ For more details see Claim 1 of \cite[Theorem 8.2]{MS}.
	\end{proof}
		
			\begin{defi} \label{d:KantUltraNorm} 
			Let $(\overline{X},d)$ be an ultra-pseudometric space and $\mathbb{F}$ be an NA valued field. Let us say that an ultra-seminorm $p$ on
			$L_{\mathbb F}(X)$ is \emph{$d$-compatible} if the pseudometric
			 induced on $\overline{X}$
			 by $p$ is $d$. We say that $p$ is  a \emph{Kantorovich ultra-seminorm for $d$}   if $p$ is the maximal  $d$-compatible ultra-seminorm on 
			  $L_{\mathbb F}(X)$.
			\end{defi}
		The maximal property of the Kantorovich norm in the classical non-discrete transportation problem was proved in \cite{MPV}, and this justifies Definition \ref{d:KantUltraNorm}.
				
				The Kantorovich ultra-norm $||\cdot||$  in Theorem \ref{t:AE} serves the NA transportation problem described in Section \ref{s:NATP}. To see this observe that one of the reformulations of this ultra-norm 
			 	($m=n$ in Claim $3$ below) coincides with the infimum  in Formula (\ref{mul:inf}) above.
			 Moreover,   using the description of the Kantorovich ultra-norm one can obtain an \emph{unbalanced} version of NATP, that is, the  case  $u=\sum\limits_{i=1}^{n}\lambda_ix_i \notin L^0_\mathbb F(X).$ 
		
						The classical analogue of the following Theorem \ref{t:AE} is the Arens-Eells embedding 
						\cite{AE}. Its usual verification is based on the dual space, 
						 involving the space of Lipschitz functions
						\cite{Flood,Wea,GaoPest}. 
				In our case the approach is different.
					If $d$ is a metric on $X$,
					  then the Kantorovich seminorm defined on
					$L^0(X)$
					is, in fact, a norm. This fact relies on the classical Hahn-Banach theorem (see \cite[Corollary 2.2.3]{Wea}) which does not always hold for general NA Banach spaces,
					 \cite{Schn,PS}. The proof that the ultra-seminorm in the following theorem  is an ultra-norm  uses only the fact that the valuation of $\mathbb F$ is NA.
					Below, in Corollary \ref{c:min}, we show  that
	the infimum in this theorem is, in fact, a minimum.
	
			\begin{thm} \label{t:AE} \emph{(Non-archimedean Arens-Eells embedding)}
				
				Let $(\overline{X},d)$ be an ultra-pseudometric space and $\mathbb{F}$ be an NA valued field.
				\begin{enumerate}
				\item There exists a Kantorovich  ultra-seminorm  $||\cdot||:=||\cdot||_d$ on  $L_\mathbb{F}(X)$ for $d$.
				Furthermore, if $d$ is an ultra-metric then $||\cdot||_d$ is an ultra-norm.
					\item
$||u||$ can be computed on the support of $u$
					for every
					 $u \in L_\mathbb{F}(X)$. That is,
					$$||u||=\inf\bigg \{\max \limits_{1\leq i\leq k}| s_i|d(x_i,y_i): u=\sum\limits_{i=1}^{k} s_i(x_i-y_i), \ x_i,y_i\in \operatorname{supp}(u), \ s_i\in \mathbb F \bigg \}.$$
					\item Moreover, if $u=\sum\limits_{i=1}^n \lambda_i x_i$ (normal form) then
					$$||u||=\inf\bigg \{\max \limits_{1\leq i,j \leq m}|c_{ij}|d(x_i,x_j): c_{ij} \in \mathbb F,   \ \forall i:  1\leq i\leq n \ \sum\limits_{j=1}^{m}c_{ij}- \sum\limits_{j=1}^{m}c_{ji}=\lambda_i  \bigg\},$$
					where $c_{ii}=0_{\mathbb F}$  and $m=|\operatorname{supp}(u)|$ (see Notation \ref{notation}).
					\end{enumerate}
			\end{thm}
	\begin{proof}
			For $u\in L_\mathbb{F}(X)$ define
		$$||u||:=\inf\bigg \{\max \limits_{1\leq i\leq n}|\lambda_i|d(x_i,y_i): u=\sum\limits_{i=1}^{n}\lambda_i(x_i-y_i), \ x_i,y_i\in \overline{X}, \ \lambda_i\in \mathbb F \bigg \} .$$
\vskip 0.3cm

\noindent \textbf{Claim 1:} $||\cdot||$ is an ultra-seminorm on $L_\mathbb{F}(X).$
\begin{proof}
Clearly, $||u||\geq 0$  for every $u\in L_\mathbb{F}(X).$
Since $\textbf{0}=\textbf{0}-\textbf{0}$ we also have $||\textbf{0}||\leq d(\textbf{0},\textbf{0})=0$ and thus $||\textbf{0}||=0.$	 The equality $||\lambda u||=|\lambda||| u||$ follows from the fact that for every $\lambda\neq 0_\mathbb F$, if $u=\sum\limits_{i=1}^n \lambda_i(x_i-y_i)$ then $\lambda u=\sum\limits_{i=1}^n \lambda\lambda_i(x_i-y_i)$ and, if
$\lambda u=\sum\limits_{i=1}^n \lambda_i(x_i-y_i)$ then $ u=\sum\limits_{i=1}^n \lambda^{-1}\lambda_i(x_i-y_i).$
Of course, we 
also use axiom (4) in the definition of valuation.
\\ Finally, observe that
$$||u+v|| \leq \max \{||u|| , ||v|| \} \ \ \ \forall \ u,v \in L_\mathbb{F}(X).$$ 
Indeed, assuming the contrary, there exist decompositions
$$u=\sum\limits_{i=1}^{k} \lambda_i(x_i-y_i), \ v=\sum\limits_{i=k+1}^{l} \lambda_i(x_i-y_i)$$
  such that
$$||u+v||>c:=\max\{\max\limits_{1\leq i\leq k}|\lambda_i|d(x_i,y_i),\max\limits_{k+1\leq i\leq l}|\lambda_i|d(x_i,y_i)\}.$$ This contradicts the definition of
$||u+v||$ since  $u+v=\sum\limits_{i=1}^{l} \lambda_i(x_i-y_i)$ with $$||u+v||>\max\{\max\limits_{1\leq i\leq k}|\lambda_i|d(x_i,y_i),\max\limits_{k+1\leq i\leq l}|\lambda_i|d(x_i,y_i)\}=\max\limits_{1\leq i\leq l}|\lambda_i|d(x_i,y_i).$$
	\end{proof}
\noindent \textbf{Claim 2:} For every
 $u \in L_\mathbb{F}(X)$
the value of $||u||$ can be computed on the support of $u.$  That is, 
	$$||u||=\inf\bigg \{\max \limits_{1\leq i\leq k}| s_i|d(x_i,y_i): u=\sum\limits_{i=1}^{k} s_i(x_i-y_i), \ x_i,y_i\in \operatorname{supp}(u), \ s_i\in \mathbb F \bigg \}.$$
\begin{proof}
		Let $u=\sum\limits_{i=1}^{k} s_i(x_i-y_i)$ be a decomposition of $u\in L_\mathbb{F}(X). $	
 Consider the following  steps which do not increase the value of $\max \limits_{1\leq i\leq k}|s_i|d(x_i,y_i)$:

\begin{enumerate}
\item Delete any term of $u$ of the form $0_\mathbb F(x-y)$ or $s_i(x-x).$
\item Replace
the term $s_i(x_i-y_i)$ with $-s_i(y_i-x_i).$
\item Assume there exist $1\leq i_0\leq n$ and $\mathbf{0}\neq z\notin \operatorname{supp}(u)$ such that $z=x_{i_0}$ or $z=y_{i_0}.$ Using steps $(1)-(2)$ we may assume without
loss of generality that the terms $\lambda(x-z)$ and $\mu(z-y)$ appear in the decomposition of $u= \sum\limits_{i=1}^{k}s_i (x_i-y_i)$
with $|\lambda|\leq |\mu|.$ Replace them with $\lambda(x-y)$ and $(\mu-\lambda)(z-y).$

This way the number of terms in which the element $z$ appears decreases. The value of the corresponding maximum  $\max \limits_{1\leq j\leq k}|\mu_j|d(x_j,y_j)$ does not increase under such a substitution, because
$$\max \{|\lambda|d(x,y),|\mu-\lambda|d(z,y)\}\leq \max \{|\lambda|d(x,z),|\mu|d(z,y)\}.$$ Indeed, using the strong triangle inequality and the fact that $|\lambda|\leq |\mu|$  we obtain $$|\lambda|d(x,y)\leq \max \{|\lambda|d(x,z),|\lambda|d(z,y)\}\leq \max \{|\lambda|d(x,z),|\mu|d(z,y)\}.$$

Also, assuming that  $|\mu-\lambda|d(z,y)>|\mu|d(z,y)$  we obtain  $|\mu-\lambda|>\max\{|\lambda|,|\mu|\},$ which contradicts the strong triangle inequality. Thus, $|\mu-\lambda|d(z,y)\leq |\mu|d(z,y).$
 Applying finitely many substitutions of this form and taking into account that the sum of  $z's$ coefficients in any decomposition of $u$ is equal to zero,
 we obtain a decomposition of $u$ with only two terms in which $z$ appears: $\lambda(x-z)$ and $\lambda(z-y).$
 These terms can be replaced by the single term  $\lambda(x-y)$ since $\lambda(x-z)+\lambda(z-y)=\lambda(x-y)$ and
 $|\lambda|d(x,y)\leq  \max \{|\lambda|d(x,z),|\lambda|d(z,y)\}.$ Now the term  $\lambda(x-y)$  and all other terms in the new decomposition
 do not contain the element $z.$
 \item  Assume there exist $1\leq i_0\leq n$ and $z=\mathbf{0}\notin \operatorname{supp}(u)$ such that $z=x_{i_0}$ or $z=y_{i_0}.$ We claim that similar to
case $(3)$ it suffices to consider  decompositions $u= \sum\limits_{i=1}^{k}s_i (x_i-y_i)$  that contain terms of the form  $\lambda(x-z)$ and $\mu(z-y)$
with $|\lambda|\leq |\mu|.$ Indeed, since $z=\mathbf{0}\notin \operatorname{supp}(u)$  it follows that $u=\sum\limits_{i=1}^{n}\lambda_it_i \in  L^{0}_\mathbb F(X)$ (normal form) and $\sum\limits_{i=1}^{n} \lambda_i=0_{\mathbb F}.$
If there exists only one term $\lambda(x-z)$ in which $z$ appears then by the previous steps we can assume, without loss of generality, that we are dealing with a decomposition of $u$ of the form $u=\sum\limits_{j=1}^{k}\mu_j(a_j-b_j)+\lambda(t_1-z),$ where $a_j,b_j\in\operatorname{supp}(u).$ On the one hand,
$u'=\sum\limits_{j=1}^{k}\mu_j(a_j-b_j)\in L^{0}_\mathbb F(X).$ On the other hand, for every $i\neq 1$ the sum of the coefficients of $t_i$ in this  decomposition of $u'$ is equal to $\lambda_i.$
It follows that the sum of $t_1's$ coefficients in $u'$ is $\lambda_1$ and this implies that $\lambda=0_{\mathbb F}.$
So we may assume that the terms $\lambda(x-z)$ and $\mu(z-y)$ appear in the decomposition of $u= \sum\limits_{i=1}^{k}s_i (x_i-y_i)$
with $|\lambda|\leq |\mu|.$ If $\lambda=\mu$ we simply replace these terms with the single term $\lambda(x-y).$  Otherwise, replace these terms with  $\lambda(x-y)$ and $(\mu-\lambda)(z-y).$ In any case we can say that completely similar to reduction  $(3)$  the number of terms in which the element $z$ appears decreases. The value of the corresponding maximum  $\max \limits_{1\leq j\leq k}|\mu_j|d(x_j,y_j)$ does not increase under such a substitution. We apply finitely many substitutions of this form and obtain a decomposition of $u$ in which all terms do not contain the element $z.$ 
 \end{enumerate}
Using  reductions $(3)$ and $(4)$   we complete the proof of Claim 2.  
\end{proof}\vskip 0.2cm

\noindent \textbf{Claim 3:} For $u=\sum\limits_{i=1}^{n}\lambda_ix_i \in L_\mathbb F(X)$ let $m=|\operatorname{supp}(u)|$
(by Notation \ref{notation} we have $m =n,$ or $m=n+1$).
Then,\begin{equation}\label{eq:claim3}
||u||=\inf\bigg \{\max \limits_{1\leq i,j \leq m}|c_{ij}|d(x_i,x_j): c_{ij} \in \mathbb F,   \ \forall i: 1\leq i\leq n \ \sum\limits_{j=1}^{m}c_{ij}- \sum\limits_{j=1}^{m}c_{ji}=\lambda_i  \bigg\},\end{equation}
 where $c_{ii}=0_{\mathbb F}.$
\begin{proof}
	By Notation \ref{notation}, $\sum\limits_{i=1}^{n}\lambda_ix_i$ is a normal form of $u$. It follows that
	 a matrix $(c_{ij})\in \mathbb F^{m\times m}$ satisfies the  equations  $$\sum\limits_{j=1}^{m}c_{ij}- \sum\limits_{j=1}^{m}c_{ji}=\lambda_i \ \forall i:  1\leq i\leq n$$ if and only if $u=\sum\limits_{i=1}^{m}\sum\limits_{j=1}^{m}c_{ij}(x_i-x_j).$ Indeed, on the one hand  the coefficient of $x_i$ in the 
	 right expression is  $\sum\limits_{j=1}^{m}c_{ij}- \sum\limits_{j=1}^{m}c_{ji}$  for all $1\leq i\leq n$. On the other hand,
the coefficient of $x_i$ in $u$ is $\lambda_i.$
Note that by our convention if $m=n+1$ then $x_{n+1}=\mathbf{0}.$  Since $d(x_i,x_i)=0$ and $c_{ii}-c_{ii}=0_{\mathbb F}$  we may assume without loss of generality that $c_{ii}=0_{\mathbb F}.$

  By Claim $2$, $||u||$ can be computed on the support of $u.$  If we have two terms of the form $\lambda(x_i-x_j),\ \mu(x_i-x_j)$ we can replace them with the single term $(\lambda+\mu)(x_i-x_j)$ since $|\lambda+\mu|\leq \max\{|\lambda|,|\mu|\}.$
Thus, we may consider only decompositions  of the form $u=\sum\limits_{i=1}^{n}\sum\limits_{j=1}^{n}c_{ij}(x_i-x_j),$
where $\sum\limits_{j=1}^{m}c_{ij}-  \sum\limits_{j=1}^{m}c_{ji}=\lambda_i$ and $c_{ii}=0_{\mathbb F}$.
\end{proof}

\vskip 0.2cm
\noindent \textbf{Claim 4:}   For $u=\sum\limits_{i=1}^{n}\lambda_ix_i \in L_\mathbb F(X)$ let $m=|\operatorname{supp}(u)|.$ Then, $$||u||\geq r\cdot l_0 $$ where $r=\max\{|\lambda_i|: 1\leq i\leq n\}$ and $l_0=\min\{d(x_i,x_j):\ 1\leq i\neq j\leq m\}.$
 \begin{proof}
 	Assuming the contrary let $||u||<r\cdot l_0.$ By Claim 3 there exists  a matrix $(c_{ij})\in\mathbb F^{m\times m}$ such that $ \sum\limits_{j=1}^{m}c_{ij}- \sum\limits_{j=1}^{m}c_{ji}=\lambda_i \  \forall i:  1\leq i\leq n \ $ and in addition
 	 $r\cdot l_0>\max\limits_{1\leq i, j\leq m}|c_{ij}|d(x_i,x_j)$.
 	  Taking into account
 	  the definition of $l_0$ we get $r>|c_{ij}| \ \forall i,j.$
 	 By the definition of $r$ there exists $1\leq i_0\leq n$ such that $r=|\lambda_{i_0}|.$ Thus, $|\lambda_{i_0}|>|c_{ij}| \ \forall i, j.$
 	 In particular,   $$|\lambda_{i_0}|>\max\{\max\limits_{1\leq j\leq m} |c_{i_{0}j}|,\max\limits_{1\leq j\leq m}|c_{ji_0}|\}.$$
 	 Applying the strong triangle inequality to the equation   $\sum\limits_{j=1}^{m}c_{i_0j}- \sum\limits_{j=1}^{m}c_{ji_0}=\lambda_{i_0}$ we obtain the contradiction  $$|\lambda_{i_0}|\leq \max\{\max\limits_{1\leq j\leq m} |c_{i_{0}j}|,\max\limits_{1\leq j\leq m}|c_{ji_0}|\}.$$
 	\end{proof}
 \noindent \textbf{Claim 5:} $\iota: (\overline{X},d) \hookrightarrow
(L_\mathbb{F}(X),||\cdot||), \
\iota(x)=\{x\}$
is an isometric embedding, i.e.
$$||x-y||=d(x,y) \ \ \ \forall \ x,y \in \overline{X}.$$
\begin{proof}
If $x=y$ the assertion is trivial so we may assume that $u=x-y\neq \textbf{0}.$  By Claim $2$ the value  $||x-y||$ can be computed on the support $\{x,y\}.$
Using also some of the reductions we mentioned above, it suffices to consider only the trivial decomposition $u=x-y.$ It follows that $||x-y||=d(x,y).$
\end{proof}

 \noindent	\textbf{Claim 6:} 
 $||u||=0$ if and only if $u$ admits a presentation 
 $u=\sum\limits_{k=1}^{t} s_k(x_k-y_k)$ such that $x_k,y_k \in \operatorname{supp}(u)$ and $d(x_k,y_k)=0$ for every $k \in \{1, \dots, t\}$. In particular, the ultra-seminorm $||\cdot||$ is an ultra-norm on $L_\mathbb{F}(X)$ if and only if $d$ is an ultra-metric on $X$. 
 \begin{proof} 
 The ``if" part is trivial. 
 	
 The ``only if" part is obvious for $u=\textbf{0}$. Suppose that $u \neq \textbf{0}$ and let $u=\sum\limits_{i=1}^{n}\lambda_ix_i$ be a normal form of $u.$
 	First suppose that $u$ is \emph{$d$-irreducible} in the following sense: 
 	there are no $1\leq i \neq j\leq m$ such that $d(x_i,x_j)=0$, where  $m=|\operatorname{supp}(u)|.$ 
 	We claim that $||u|| >0$. Indeed, 
 	the corresponding $l_0$ defined in Claim 4 is positive and 
 	we have  $||u||\geq r\cdot l_0$. Clearly, $r >0$ because  $u \neq \textbf{0}$. So, 
 	we get that $||u|| >0$. 
 	
 	Now we can suppose that $u$ is $d$-reducible. 
 	 We describe a certain reduction for $u$. Choose a pair $i \neq j$ such that $d(x_i,x_j)=0$. Without loss of generality we may assume that $x_i\neq\textbf{0}.$ Denote $w_1:= \lambda_i(x_i-x_j), u_1:=u-w_1$. By Claims 1 and 5 we know that $||w_1||=0$. Hence, $||u||=||u-w_1||=0$.
 	 
 	 $\blacktriangleright$ In case  $x_j=\textbf{0}$   
 	  delete the term $\lambda_ix_i$ in the presentation $u=\sum\limits_{i=1}^{n}\lambda_ix_i$ to obtain a normal form of $u-w_1.$
 	  
 	   $\blacktriangleright$  In case $x_j\neq \textbf{0}$ observe that
 	$\lambda_ix_i + \lambda_j x_j=\lambda_i (x_i-x_j) + (\lambda_i+\lambda_j)x_j$.  Replacing the terms $\lambda_ix_i, \lambda_j x_j$ in the presentation $u=\sum\limits_{i=1}^{n}\lambda_ix_i$ with the single term $(\lambda_i+\lambda_j)x_j$ we get a normal form of $u-w_1$. 
 	
 	In both cases we can then use the same reductions for $u_1$ to obtain $u_2:=u_1-w_2$, etc. Continuing in this manner  we get, after finitely many steps, a vector $u_t$ such that $||u||=||u_t||=0$ and in the normal presentation of $u_t$ we have no pair of distinct elements $a,b \in X$ such that $d(a,b)=0$. 
 	That is, $u_t$ is $d$-irreducible. 
 	 Then necessarily $u_t=\textbf{0}$. Indeed, 
 	if not, then as above we obtain that $||u_t||>0$. 
 			
 		So, $u_t=\textbf{0}$. Hence, $u=\sum\limits_{k=1}^{t} w_k$. By the definition of $w_k$ this proves Claim 6. 
 \end{proof}

\vskip 0.3cm

\noindent \textbf{Claim 7:} (Maximality property) Let $\sigma$ be an ultra-seminorm on $L_\mathbb{F}(X)$
such that
\begin{equation} \label{eq:sigma}
\sigma(x-y) \leq d(x,y)\ \ \forall x,y\in \overline{X}.
\end{equation}
 Then
$\sigma\leq ||\cdot||.$
\begin{proof}
Let $u$ be a non-zero element of $L_\mathbb{F}(X)$ and
$\sigma$ 
be an ultra-seminorm  which satisfies
(\ref{eq:sigma}).
Then for every decomposition $u=\sum\limits_{i=1}^{n}\lambda_i(x_i-y_i), \ x_i,y_i\in \overline{X}$  we obtain
$$\sigma(u)=\sigma(\sum\limits_{i=1}^{n}\lambda_i(x_i-y_i))\leq \max_{1 \leq i \leq n}|\lambda_i|\sigma(x_i-y_i) \leq  \max_{1 \leq i \leq n}|\lambda_i|d(x_{i},y_{i}).$$  It follows from the definition of the ultra-seminorm $||\cdot||$  that $\sigma(u) \leq ||u||.$
		\end{proof}
		
		Combining the claims we complete the proof of Theorem \ref{t:AE}.
\end{proof}
\begin{example}
	Let $\mathbb F:=\mathbb Z_2$ be the discrete field of two elements. Note that in this case $(L_\mathbb{F}(X), ||\cdot||)$, as a topological group,  coincides with $B_{\scriptscriptstyle\mathcal{NA}}$  the \emph{uniform free NA Boolean group} over $(X,d).$  Indeed, this follows from the fact that  $B_{\scriptscriptstyle\mathcal{NA}}$ is metrizable by a Graev type ultra-norm (see \cite{MS}).
\end{example}

\begin{remark}  \label{rem:addthird}   \
	 Theorem \ref{t:attaining}  shows that in Theorem \ref{t:AE} we can assume, in addition, that:
	\begin{enumerate}
	\item The infimum  in  Theorem \ref{t:AE} is attained.
	\item The coefficients $c_{ij}$ (in Theorem \ref{t:AE}.3) belong to the
	additive subgroup $G_u$ of $\mathbb F,$ generated by the normal coefficients $\lambda_i$ of $u$.
	\end{enumerate}
\end{remark}

\begin{remark} Using Claim $3$ and additional computations we obtain a simplified version of
	Equation (\ref{eq:claim3}):
$$||u||=\min \bigg \{\max \limits_{1\leq i<j\leq m}|c_{ij}|d(x_i,x_j):
\forall i\geq j \ c_{ij}=0 \ , \ \forall i:  1\leq i\leq n   \ \sum\limits_{j=i+1}^{m}c_{ij}- \sum\limits_{j=1}^{i-1}c_{ji}=\lambda_i \bigg\}.$$
\end{remark}

 \section{Generalized integer value property}

\subsection{$G$-value property for subgroups $G \subseteq \Bbb R$}

 First recall the {\it integer value property} for the case $\mathbb F=\Bbb R$.
  Let $d$ be a (pseudo)metric on $X$ and $||\cdot||$ 
  be its Kantorovich (semi)norm.
 For an element of   $L^0 (X)$ with integer coefficients  the inf-sum cost Formula (\ref{classform}) achieves its infimum at
an integer matrix $(c_{ij}).$
  See, for example, Sakarovitz \cite[p. 179]{Sak}, and
  Uspenskij \cite{Us-free}.

 Replacing the
 group of integers $\Bbb Z$ with any other additive subgroup  $G$ of $\Bbb R$ we obtain a natural generalization. We call it
  the \emph{$G$-value property}. It means that whenever we have an element of  $L^0 (X)$ with coefficients from $G$, the minimum in the formula is obtained at a matrix with elements from $G.$
 This generalized version  can be proved using
  the tools of convex analysis as in \cite{Us-free}.

  In the sequel we prove the $G$-value property for the NA case.

\vskip 0.3cm

\subsection{$G$-value property in the non-archimedean case}

In this subsection
let $\mathbb F$ be an NA valued field and $(X,d)$ be an ultra-(pseudo)metric space.

\begin{lemma} \label{l:GenDigital}
	Let $G$ be an additive subgroup of an NA valued field  $\mathbb F.$
	Let $u= \sum\limits_{i=1}^{n}\lambda_i x_i\in L_\mathbb{F}(X)$
	with $\lambda_i\in  G \ \forall i.$
	Then the ultra-seminorm $||u||$ can be computed using only the coefficients from $G$. That is, in the  formula of Theorem \ref{t:AE}.2 we get
	$$||u||:=\inf\bigg \{\max \limits_{1\leq k\leq l} |\rho_k| d(s_k,t_k): u=\sum\limits_{k=1}^{l}\rho_k(s_k-t_k), \ s_k,t_k\in \overline{X}, \ \rho_k\in G \bigg \} .$$
\end{lemma}
	\begin{proof}
		It is equivalent to show that for every decomposition $u=\sum\limits_{j=1}^{m}\mu_j (a_j-b_j)$  there exists
		a decomposition $u=\sum\limits_{k=1}^{l}\rho_k(s_k-t_k)$ with $\rho_k\in  G \ \forall k:  1\leq k \leq l$ such that $$\max \limits_{1\leq k\leq l} |\rho_k| d(s_k,t_k) \leq \max \limits_{1\leq j\leq m} |\mu_j| d(a_j,b_j).$$
		
		 By  deleting any term of $u$ of the form $\mu_j(x-x)$ we may assume that $a_j\neq b_j \ \forall j.$   If $\mu_j\in G \ \forall j:  1\leq j\leq m$ there is nothing to prove. So, without loss of generality, we may assume that $\mu_1\notin G$.
		
		 Moreover we can suppose that $a_1\neq \mathbf{0}$ (otherwise, write the summand $(-\mu_1)(b_1-a_1)$ instead of $\mu_1(a_1-b_1)$).
		 Consider the set of indices $$A:=\{j\neq 1: a_j=a_1 \vee b_j=a_1\}.$$
		 We show that there exists $j\in A$ such that $\mu_j\notin G.$ If $a_1\in \operatorname{supp}(u)$ then there exists $1\leq i\leq n$ such that  $a_1=x_i.$ Hence,
		 $$\mu_1+\sum\limits_{j\in A}{k_j}\mu_j=\lambda_i$$
		 where $k_j=1$ if $a_j=a_1$ and $k_j=-1$ if $b_j=a_1.$
	If  $a_1\notin \operatorname{supp}(u)$ then  $$\mu_1+\sum\limits_{j\in A}{k_j}\mu_j=0_\mathbb{F}.$$  Since $G$ is an additive subgroup of  $\mathbb F, \ \mu_1\notin G$ and $\{0_{\mathbb{F}},\lambda_i\}\subseteq G$, we conclude  that  there exists $j\in A$ such that $\mu_j\notin G.$
	
	Since $|\mu_j|=|-\mu_j|, |\mu_1|=|-\mu_1|$ we may assume, without loss of generality,  that there exists $j\neq 1$ such that $b_j=a_1, \ \mu_j\notin G$ and $|\mu_j|\leq |\mu_1|.$ Replace the terms $\mu_1(a_1-b_1)$ and $\mu_j (a_j-a_1)$
with $\mu_j(a_j-b_1)$ and $(\mu_1-\mu_j)(a_1-b_1).$ 
We show that $$\max \{|\mu_j|d(a_j,b_1),|\mu_1-\mu_j|d(a_1,b_1)\}\leq \max \{|\mu_j|d(a_j,a_1),|\mu_1|d(a_1,b_1)\}.$$ This way we decrease the number of terms in which the element $a_1$ appears with scalar coefficient not from $G.$ Since $|\mu_j|\leq |\mu_1|$ it follows from the strong triangle inequality of the valuation $|\cdot|$ that
\begin{align*}
|\mu_1-\mu_j|d(a_1,b_1) &\leq \max \{|\mu_1|d(a_1,b_1),|\mu_j|d(a_1,b_1)\}=|\mu_1|d(a_1,b_1)\leq \\ &\leq \max \{|\mu_j|d(a_j,a_1),|\mu_1|d(a_1,b_1)\}.
\end{align*}
From  the strong triangle inequality of $d$ we obtain $$|\mu_j|d(a_j,b_1)\leq \max \{|\mu_j|d(a_j,a_1),|\mu_j|d(a_1,b_1)\}\leq\max \{|\mu_j|d(a_j,a_1),|\mu_1|d(a_1,b_1)\}.$$  Therefore,
$$\max \{|\mu_j|d(a_j,b_1),|\mu_1-\mu_j|d(a_1,b_1)\}\leq \max \{|\mu_j|d(a_j,a_1),|\mu_1|d(a_1,b_1)\}.$$ Applying finitely many substitutions of this form to
terms in which the element $a_1$ appears and  in which the coefficients are not taken from $G,$ we obtain a decomposition in which all coefficients of $a_1$
(if there are any) are from $G.$  Repeating this algorithm for other elements, if necessary,   we obtain a decomposition of the form $$u=\sum\limits_{k=1}^{l}\rho_k(s_k-t_k)$$ with $\rho_k\in  G \ \forall k:  1\leq k \leq l$ such that $$\max \limits_{1\leq k\leq l} |\rho_k| d(s_k,t_k) \leq \max \limits_{1\leq j\leq m} |\mu_j| d(a_j,b_j).$$
\end{proof}

\begin{notation}
For every $u= \sum\limits_{i=1}^{n}\lambda_i x_i\in L_\mathbb{F}(X)$ (normal form) denote by $G_u$ the additive subgroup of $\mathbb F$ generated by the coefficients $\lambda_i$ of $u$.
\end{notation}

Observe that by the strong triangle inequality for every $c \in G_u$ we have
\begin{equation} \label{eq:r}
|c| \leq r:=\max\{|\lambda_i|: 1\leq i\leq n\}.
\end{equation}

\begin{lemma} \label{cor:fromg} \emph{(NA local $G_u$-value property)}
	For every $u= \sum\limits_{i=1}^{n}\lambda_i x_i\in L_\mathbb{F}(X)$
	we have
	$$||u||=\inf\bigg \{\max \limits_{1\leq i,j \leq m}|c_{ij}|d(x_i,x_j): 	c_{ij} \in G_u, \ \ \forall i:  1\leq i\leq n  \ \sum\limits_{j=1}^{m}c_{ij}- \sum\limits_{j=1}^{m}c_{ji}=\lambda_i  \bigg\}.$$
\end{lemma}
\begin{proof}
	Combine Lemma \ref{l:GenDigital} with  Claims $2,3$ of Theorem \ref{t:AE} taking into account the following
	 observation. Let  $u=\sum\limits_{k=1}^{l}\rho_k(s_k-t_k)$ with $\rho_k\in  G \ \forall k:  1\leq k \leq l.$
Since $G$ is an additive subgroup of $\mathbb F$, each reduction appearing in the proof of Claim $2$ yields a decomposition of the same form. That is, the coefficients in the resulting  decomposition are   from $G.$
	\end{proof}

\begin{lemma} \label{thm:min}
	Let $u=\sum\limits_{i=1}^{n}\lambda_i x_i\in  L_\mathbb{F}(X)$.
	Suppose that for every positive reals $a \leq b$ the set $A_{ab}:=\{|x|:  x\in G_u, \ a\leq |x|\leq b \}$ is finite.
	Then
	$$||u||=\min \bigg \{\max \limits_{1\leq i,j\leq m}|c_{ij}|d(x_i,x_j):
	c_{ij} \in G_u, \forall i:  1\leq i\leq n \ \sum\limits_{j=1}^{m}c_{ij}- \sum\limits_{j=1}^{m}c_{ji}=\lambda_i   \bigg \}.$$
\end{lemma}	
\begin{proof} 
In case $||u||=0$ we do not need the finiteness assumption. Indeed, by  the proof  of Claim 6  of Theorem \ref{t:AE}  there exists a matrix $(c_{ij})\in  G_u^{m\times m}$ such that 
$$u=\sum\limits_{i=1}^{m}\sum\limits_{j=1}^{m}c_{ij}(x_i-x_j),$$ where for every $i,j$ either $d(x_i,x_j)=0$ or $c_{ij}=0_{\mathbb F}$. It follows that   $$\sum\limits_{j=1}^{m}c_{ij}- \sum\limits_{j=1}^{m}c_{ji}=\lambda_i \ \forall i: 1\leq i\leq n$$ and thus the infimum in 
Lemma \ref{cor:fromg} is attained. So without restriction of generality we may assume 
that $||u|| > 0$. 
	We have to show that the infimum in 
	Lemma \ref{cor:fromg} is attained.  Assuming the contrary and taking into account Formula (\ref{eq:r}),
	there exists a sequence of matrices $$\{(c_{ij}^k):k\in \Bbb N \}\subseteq G_u^{m\times m}$$ with the following properties:
	\begin{enumerate}
	\item $\forall i,j,k \ \ |c_{ij}^k| \leq r$;
	\item $\forall k\in \Bbb N \ \forall i:  1\leq i\leq n \ \ \sum\limits_{j=1}^m c_{ij}^k-\sum\limits_{j=1}^m c_{ji}^k=\lambda_i;$
	\item $\max\limits_{1\leq i,j\leq m}|c_{ij}^k|d(x_i,x_j)>\max\limits_{1\leq i,j\leq m}|c_{ij}^{k+1}|d(x_i,x_j)>||u||.$  \end{enumerate}
	Passing to a subsequence, if necessary, we can also assume that there exists a pair of indices $(i_0,j_0)$ such that  $$ \forall k\in \Bbb N \  \max\limits_{1\leq i,j\leq m}|c_{ij}^k|d(x_i,x_j)=|c_{i_0j_0}^k|d(x_{i_0},x_{j_0}).$$  
	It follows that
	$$\forall k\in \Bbb N \ \ \ r\geq |c_{i_0j_0}^k|>|c_{i_0j_0}^{k+1}|>\frac{||u||}{d(x_{i_0},x_{j_0})}>0.$$
	By our assumption the set 
	$$A=\bigg\{|x|: x\in G_u, \ r\geq |x| \geq \frac{||u||}{d(x_{i_0},x_{j_0})}\bigg\}$$ is finite. This contradicts the fact that the set $\{|c_{i_0j_0}^k|: k\in \Bbb N\}$, being a strictly decreasing sequence, is infinite.
\end{proof}

By $\operatorname{char}(\mathbb F)$ we denote the characteristic of the field $\mathbb F$. Recall that if $\operatorname{char}(\mathbb F)=0$ then the field $\Bbb Q$ of rationals is naturally embedded in $\mathbb F$. 

\begin{lemma}\label{lem:fin}
Let $(\mathbb{F},|\cdot|)$ be an NA  valued field with $\operatorname{char}(\mathbb F)=0$. Then,
for every positive reals $a \leq b$ the set $\{|q|: \ a\leq |q|\leq b, \ q\in\Bbb Q \}$ is finite.
\end{lemma}
\begin{proof}
By  Ostrowski's Theorem \ref{Ostrowski}  the restricted valuation on $\Bbb Q \subseteq \mathbb F$ is discrete. Hence, the set $\{|q|: q\in \Bbb Q\setminus \{0_\mathbb F\}\}$ is closed and discrete.
 It follows that for any positive reals $a \leq b$
 the set $\{|q|: \ a\leq |q|\leq b, \ q\in \Bbb Q \}$ is compact and discrete and thus finite.
	\end{proof}
	
	\begin{thm}[{Min-attaining Theorem}] \label{t:attaining}
		Let $(\mathbb{F},|\cdot|)$ be an NA valued field. 	Let $u=\sum\limits_{i=1}^{n}\lambda_i x_i\in  L_\mathbb{F}(X)$.	 Then, 	$$||u||=\min \bigg \{\max \limits_{1\leq i,j\leq m}|c_{ij}|d(x_i,x_j):
		c_{ij} \in G_u, \ \forall i:  1\leq i\leq n \ \sum\limits_{j=1}^{m}c_{ij}- \sum\limits_{j=1}^{m}c_{ji}=\lambda_i   \bigg \}.$$
	\end{thm}
	\begin{proof}
		We show that Lemma \ref{thm:min} can be applied to every NA valued field $(\mathbb{F},|\cdot|)$  and to every $u=\sum\limits_{i=1}^{n}\lambda_i x_i\in  L_\mathbb{F}(X)$.
		
		$\blacktriangleright$ In case  $\operatorname{char}(\mathbb F)>0$   the subgroup $G_u$ is finite, being a finitely generated additive subgroup of a field of positive characteristic. So,   it is trivial that the set $A_{ab}$ from Theorem \ref{thm:min} is finite.
		
		$\blacktriangleright$ Now assume that $\operatorname{char}(\mathbb F)=0.$ Instead  of showing directly that the set  $$A_{ab}:=\{|x|:  x\in G_u, \ a\leq |x|\leq b \}$$ is finite for every positive reals $a \leq b$, we will show that it is contained in a finite subset  $B_{ab}$ of $\Bbb R$. Let $$B_{ab}:=\{|x|: x\in \widetilde{G_u}, \ a\leq |x|\leq b\}$$ where
		 $\widetilde{G_u}:=\{\sum\limits_{i=1}^{n}m_i\lambda_i| \ m_i\in \Bbb Q\}.$  Since $G_u\subseteq \widetilde{G_u}$ we also have $A_{ab}\subseteq  B_{ab}.$ We prove the finiteness of the set $B_{ab}$ using induction on $n,$ the number of scalar coefficients $\lambda_i$ in the normal form of $u.$
		
		 First, for the case $n=1$ let $u=\lambda x.$ We show that the set $\{|m\lambda|: a\leq |m\lambda|\leq b, \ m\in \Bbb Q \}$ is finite.  It is equivalent to show that the set $\{|m|: c\leq |m|\leq d,\ m\in \Bbb Q \}$ is finite, where $c=\frac{a}{|\lambda|}, d=\frac{b}{|\lambda|}.$ This set is finite by Lemma \ref{lem:fin}.
		
		 Let $u=\sum\limits_{i=1}^{n+1}\lambda_ix_i$ and $v=\sum\limits_{i=1}^{n}\lambda_ix_i.$ By the induction hypothesis the set  $$C_{ab}:=\{|x|: x\in \widetilde{G_v}, \ a\leq |x|\leq b\}$$ is finite. If
		 $$\{|x|: x\in \widetilde{G_u}\setminus  \widetilde{G_v}, \ a\leq |x|\leq b\}=\emptyset$$ there is nothing to prove.
		
		 So we may assume that there exists an element of $\widetilde{G_u}$ of the form
		 $t=\sum\limits_{i=1}^{n+1}t_i\lambda_i,$ where $t_i\in \Bbb Q \ \forall i,$
		 $\ t_{n+1}\neq 0, \ a\leq |t|\leq b$ and $|t|\notin C_{ab}.$
		It follows from Lemma \ref{lem:fin} that the set $$D:=\{|qt|:q\in \Bbb Q, \  a\leq |qt|\leq b \}$$ is finite. It suffices to show that
		$$\{|x|: x\in \widetilde{G_u} \setminus  \widetilde{G_v}, \ a\leq |x|\leq b\}\subseteq C_{ab}\cup D.$$
		
		 Let $s=\sum\limits_{i=1}^{n+1}s_i\lambda_i\in \widetilde{G_u}\setminus  \widetilde{G_v}$ and
		 $a\leq |s|\leq b$.
		
		  We will show that $|s|\in  C_{ab}\cup D.$
		 Since $s\in \widetilde{G_u}\setminus  \widetilde{G_v}$ then $s_{n+1}\neq 0.$
		 Since  $\ t_{n+1}\neq 0$,
		 it follows that $\exists q\in \Bbb Q \setminus \{0\}$ such that $qt_{n+1}=s_{n+1}.$
		 Thus there exists $r\in \widetilde{G_v}$ such that $s=qt+r.$
		 Clearly $|qt|\neq |r|.$ Indeed, otherwise, we have $|t|=|\frac{1}{q}r|$ contradicting the fact that $|t|\notin C_{ab}.$ So, by the basic properties of the strong triangle inequality, either $|s|=|qt|\in D$ or
		 $|s|=|r|\in C_{ab}.$ Therefore   $B_{ab}\subseteq C_{ab}\cup D$, as needed.
		\end{proof}

\begin{corol} \label{c:min}
The infimum  in Theorem \ref{t:AE} is, in fact, a minimum.
\end{corol}

\begin{prop} \label{p:upper}
	For every $u=\sum\limits_{i=1}^{n}\lambda_ix_i \in L_\mathbb F(X)$ we have
	$$r\cdot l_0 \leq ||u||\leq r\cdot l_1 $$ where $r=\max\{|\lambda_i|: 1\leq i\leq n\}$, $l_1=\max\{d(x_i,x_j):\ 1\leq i,j\leq m\}$, $l_0=\min\{d(x_i,x_j):\ 1\leq i \neq j\leq m\}$ and $m=|\operatorname{supp}(u)|.$
\end{prop}
\begin{proof}
	Claim 4 of Theorem  \ref{t:AE} provides a lower bound $r\cdot l_0 \leq ||u||.$
	
	By Theorem \ref{cor:fromg}
	$$||u||=\inf \bigg \{\max \limits_{1\leq i,j\leq m}|c_{ij}|d(x_i,x_j):
	c_{ij} \in G_u,   \forall i:  1\leq i\leq n \ \sum\limits_{j=1}^{n}c_{ij}- \sum\limits_{j=1}^{n}c_{ji}=\lambda_i  \bigg\},$$
	while $|c_{ij}| \leq r$	by 
	(\ref{eq:r}). Therefore
	$||u|| \leq \max \limits_{1\leq i,j\leq m}|c_{ij}|d(x_i,x_j)\leq r\cdot l_1.$
\end{proof}

\begin{corol} \label{c:special}
	Let $u=\sum\limits_{i=1}^{n}\lambda_ix_i \in L_\mathbb F(X)$.
	Suppose that $l=d(x_i,x_j)$ for every $x_i \neq x_j \in \operatorname{supp}(u)$. Then $||u||=r \cdot l$ where $r=\max\{|\lambda_i|: 1\leq i\leq n\}$.
\end{corol}

 \section{Free NA locally convex space}\label{s:freeLCS}

For the free locally convex $\mathbb{F}$-spaces
(where $\mathbb F= \Bbb R$ or $\Bbb C$) on uniform spaces we refer to Raikov \cite{MPV}. Here we consider their NA analogue.
 Let $\mathbb{F}$ be an NA valued field.
  Recall \cite{Schn,PS} that
 a Hausdorff
 NA $\mathbb{F}$-vector space $V$ is said to be \emph{locally convex} if its topology can be generated by a family of ultra-seminorms.

 Assigning to every
 NA locally convex $\mathbb{F}$-space $V$ its uniform space $(V,{\mathcal U}),$ we define a forgetful functor
 from the category $_\mathbb{F}$LCS$_{\scriptscriptstyle\mathcal{NA}}$
 of all Hausdorff
 NA locally convex spaces to the category of all NA Hausdorff uniform spaces
 $\mathbf{Unif}_{\scriptscriptstyle\mathcal{NA}}$.

 \begin{defi} \label{d:FreeGr}
 	Let $\mathbb{F}$ be an NA valued field and
 	$(X,{\mathcal U}) \in \mathbf{Unif}_{\scriptscriptstyle\mathcal{NA}}$ be an NA uniform space. By a
 	\emph{free NA locally convex $\mathbb{F}$-space} of $(X,{\mathcal U})$ we
 	mean a pair $(L_\mathbb{F}(X,{\mathcal U}),i)$ (or, simply,
 	$L_\mathbb{F}(X,{\mathcal U})$ or $L_\mathbb{F}(X)$ when $i$ and ${\mathcal U}$ are understood), where
 	$L_\mathbb{F}(X,{\mathcal U})$ is a locally convex $\mathbb{F}$-space
 	and $i: X \to L_\mathbb{F}(X,{\mathcal U})$ is a uniform map
 	satisfying the following universal property. For every uniformly
 	continuous map \textbf{$\varphi: (X,{\mathcal U}) \to V$} into a locally convex
 	$\mathbb{F}$-space $V$, there exists a unique
 	continuous linear homomorphism \textbf{$\Phi: L_\mathbb{F}(X,{\mathcal U})
 		\to V$} for which the following diagram commutes:
 	\begin{equation*} \label{equ:ufn}
 	\xymatrix { (X,{\mathcal U}) \ar[dr]_{\varphi} \ar[r]^{i} & L_\mathbb{F}(X,{\mathcal U})
 		\ar[d]^{\Phi} \\
 		& V }
 	\end{equation*}
 \end{defi}

 A categorical reformulation of this definition is that
 $i: X \to L_\mathbb{F}(X,{\mathcal U})$ is a universal arrow from $(X,{\mathcal U})$
 to the forgetful functor
 $_\mathbb{F}$LCS$_{\scriptscriptstyle\mathcal{NA}} \to \mathbf{Unif}_{\scriptscriptstyle\mathcal{NA}}$.
 The uniformity $\overline{\mathcal U }$ in the following theorem is obtained from the uniformity $\mathcal{U}$ by adding to $X$ the element $\mathbf{0}$ as an  isolated point. In particular, if $\mathcal U$ is metrizable and $d$ is the corresponding ultra-metric,  one can extend $d$ from $X$ to
 $\overline{X}$ such that $d$ induces the uniformity $\overline{\mathcal U}$
 (apply Lemma \ref{l:extend}).

 \begin{thm} \label{t:freeLCS}
 	For every Hausdorff NA uniform space $(X,{\mathcal U})$
 	the uniform NA free locally convex $\mathbb{F}$-space exists. Its structure can be defined as follows. Let $D$ be the set of all 
 	 $\overline{\mathcal U}$-uniformly continuous ultra-pseudometrics on $\overline{X}:=X\cup \{\mathbf{0}\}$. For every $d \in D$ we have the corresponding Kantorovich ultra-seminorm $||\cdot||_d$ on $L_\mathbb{F}(X).$
 	Then $L_\mathbb{F}(X)$ endowed with the family $\Gamma:=\{||\cdot||_d: d \in D\}$ of Kantorovich ultra-seminorms defines the desired uniform NA free locally convex $\mathbb{F}$-space which we denote by $L_\mathbb{F}(X,{\mathcal U})$.  The corresponding arrow $i: (X,{\mathcal U})\to L_\mathbb{F}(X,{\mathcal U})$ is a uniform embedding.
 \end{thm}
 \begin{proof}
 	First of all, observe that $L_\mathbb{F}(X,{\mathcal U})$ is Hausdorff. Indeed, this follows by analyzing Claims 4 and 6 of Theorem \ref{t:AE} (or, Proposition \ref{p:upper}).
 	
 	Next we have the following commutative diagram
 	\begin{equation*} \label{e:K1}
 	\xymatrix { (X,{\mathcal U}) \ar[dr]_{\varphi} \ar[r]^{i} & L_\mathbb{F}(X,{\mathcal U})
 		\ar[d]^{\Phi} \\
 		& V }
 	\end{equation*}
 	Now we only have  to show that $\Phi$ is continuous.
 	Since $\mathbf{0}$ is isolated in $(\overline{X},\overline{\mathcal U } )$ and $\varphi: (X,{\mathcal U})\to V$ is uniformly continuous, so is the natural extension  $\varphi:(\overline{X},\overline{\mathcal U })\to V.$    By our assumption $V$ has a family $\Gamma_V$ of ultra-seminorms which generate its topology. Every $\rho \in \Gamma_V$ induces
 	an ultra-seminorm $\sigma_{\rho}$ on $L_\mathbb{F}(X)$
 	and
 	an ultra-pseudometric $d_{\rho}$ on $\overline{X}$ defined by
 	$$
 	\sigma_{\rho}(u):=\rho(\Phi(u)), \ \ \
 	d_{\rho}(x,y):=\rho(\varphi(x)-\varphi(y)),
 	$$
 	respectively.
 
 	Since $\varphi:(\overline{X},\overline{\mathcal U })\to V$ is uniformly continuous we have $d_{\rho}\in D.$
 	Consider the corresponding Kantorovich ultra-seminorm $||\cdot||_{d_{\rho}}$ on $L_\mathbb{F}(X)$.
 	Then $\sigma_{\rho}(x-y)={d_{\rho}}(x,y)$ for every $x,y \in \overline {X}$. By the maximality property (Definition \ref{d:KantUltraNorm} and Theorem \ref{t:AE}) we obtain  $||\cdot||_{d_{\rho}} \geq \sigma_{\rho}.$ This guarantees that
 	$\rho(\Phi(u)) \leq ||u||_{d_{\rho}}$ for every $u \in L_\mathbb{F}(X)$, which implies
 	the continuity of $\Phi$.
 	
 	Finally, note that by Lemma \ref{l:extend} and Theorem \ref{t:AE} the family $\Gamma$ of Kantorovich  ultra-seminorms generates the original uniform structure ${\mathcal U}$ on $X=i(X) \subseteq L_\mathbb{F}(X)$. Hence $i$ is a uniform embedding.
 \end{proof}

 \begin{prop} \label{p:emb}
 	Let $\mathbb F$ be an NA valued filed and $K$ 
 	a subfield of $\mathbb F.$ Then for every
 	Hausdorff NA uniform space $(X,{\mathcal U})$
 	the natural algebraic inclusion $j: L_K(X) \to L_{\mathbb F}(X)$ of $K$-vector spaces is a topological embedding.
 \end{prop}
 \begin{proof} Let $d$ be a uniformly continuous ultra-pseudometric on $\overline{X}:=X\cup \{\mathbf{0}\}$.
 	Denote by $||\cdot||^K$ and $||\cdot||^{\mathbb F}$  the corresponding Kantorovich ultra-seminorms
 	of $d$ in $L_K(X)$ and $L_{\mathbb F}(X)$ respectively. Let
 	$u= \sum\limits_{i=1}^{n}\lambda_i x_i\in L_K(X) \subseteq L_{\mathbb F}(X)$.
 	Then clearly $G_u$ is an additive subgroup of $K$ and of $\mathbb F.$ Therefore by Theorem \ref{cor:fromg} we have $||u||^K=||u||^{\mathbb F}$.
 	Now Theorem \ref{t:freeLCS} guarantees that $j: L_K(X) \to L_{\mathbb F}(X)$ is a topological embedding.
 	\end{proof}

As in the classical case of the fields $\Bbb R$ or $\Bbb C$ (see \cite{Raikov}) we have the following property for the NA case. 

 \begin{prop}
 	\label{closed} The universal arrow $i: (X,{\mathcal U})\to L_\mathbb{F}(X,{\mathcal U})$ is a
 	closed embedding for any NA valued field $\mathbb{F}$.
 \end{prop}
 \begin{proof} We have to show that $X=i(X)$ is closed in $L_\mathbb{F}(X)$. Let $v \in L_\mathbb{F}(X)$ be a vector such that $v \notin X$. It is enough to find a locally convex space $V$ and a continuous linear morphism $\Phi: L_\mathbb{F}(X) \to V$ such that $\Phi(v) \notin cl(\Phi(X))$.
  For $v=\lambda x$ with $\lambda \neq 1$ and $x \in X$ consider the continuous functional $$\Phi: L_\mathbb{F}(X) \to \mathbb{F},  \ \sum_{k=1}^m \lambda_k x_k \mapsto \sum_{k=1}^m \lambda_k.$$
 Then $\Phi(v) = \lambda \notin cl(\Phi(X))=\{1\}$. The same $\Phi$ works for the case of $v={\mathbf 0}$.
 	
 	Now we may suppose that $v=\sum_{i=1}^n \lambda_i x_i$ with non-zero coefficients $\lambda_i$ and that $\operatorname{supp}(v)$ contains at least two elements from $X$. That is, $\operatorname{supp}(u)=\{x_1,x_2,x_3, \ldots, x_n\}$, where $x_1, x_2 \in X$  and  $n \geq 2$.
 	Define $V$ as the 2-dimensional NA normed $\mathbb{F}$-space $\mathbb{F}^2$
 	(with the $\max$ ultra-norm).
 	Since the uniform space  $(X,{\mathcal U})$ is NA and Hausdorff, one may partition it into three clopen disjoint subsets
 	$$
 	X = X_1 \cup X_2 \cup X_3
 	$$
 	such that $$x_1 \in X_1, x_2 \in X_2, x_k \in X_3 \ \ \forall \ 3\leq  k \leq n.$$
 	Now define
 	$$
 	\varphi: X \to V=\mathbb{F}^2, \ \ \ \
 	\varphi(x) =
 	\begin{cases}
 	(1,0) & {\text{for}} \ x \in X_1\\
 	(0,1) & {\text{for}} \ x \in X_2\\
 	(0,0) & {\text{for}} \ x \in X_3.
 	\end{cases}
 	$$
 	This map is uniformly continuous and $\mathbb{F}^2$ is a locally convex NA $\mathbb{F}$-space. Hence, by the universality property, there exists the continuous extension $\Phi: L_\mathbb{F}(X) \to V$. Now observe that $$\Phi(v)=(\lambda_1,\lambda_2) \notin cl(\Phi(X))=\{(1,0), (0,1), (0,0)\}.$$	
 \end{proof}

 \subsection{Normability and metrizability}

 \begin{thm} \label{t:normable}
 	Let $\mathbb F$ be an NA valued field with a trivial valuation, $(X,d)$ be an ultra-metric space and ${\mathcal U}(d)$ 
 	be the uniformity of $d$. Then the free NA locally convex space $L_\mathbb{F}(X,{\mathcal U}(d))$ is normable by the Kantorovich ultra-norm $||\cdot||_d$.
 \end{thm}
 \begin{proof}
 	As in Lemma \ref{l:extend} consider the extension of $d$ on $\overline{X}$. Next, by
 	Theorem \ref{t:AE}, we have the corresponding Kantorovich ultra-norm $||\cdot||.$  It suffices to show that if $\varphi: (X,d) \to V$ is a uniformly continuous map to a locally convex space $V,$ then the linear extension $ \Phi: (L_{\mathbb F}(X),||\cdot||)\to V$ is continuous.
 	Being a locally convex space the topology of $V$ is defined by a collection of ultra-seminorms $\{\rho_i\}_{i\in I}.$ Clearly, $\varphi: (\overline{X},d)\to V$ is uniformly continuous. Fix $\varepsilon>0$ and $i_0\in I.$ It follows that there exists
 	$\delta>0$ such that $\rho_{i_0} (\varphi(x)-\varphi(y))<\varepsilon \ \forall x,y\in \overline{X} $ with $d(x,y)<\delta.$ Now assume that $u\in L_{\mathbb F}(X)$ with $||u||<\delta.$ We prove the continuity of $\Phi$ by showing that $\rho_{i_0}(\Phi(u))<\varepsilon.$  By the definition of the ultra-norm $||\cdot||$ there exists a decomposition $u=\sum\limits_{i=1}^{n}\lambda_i(x_i-y_i)$ such that $\max\limits_{1\leq i\leq n}|\lambda_i|d(x_i,y_i)<\delta.$ Since the valuation $|\cdot|$ is trivial we obtain  $\max\limits_{1\leq i\leq n}d(x_i,y_i)<\delta.$
 	It follows that
 	\begin{align*}
 &\rho_{i_0}(\Phi(u))=\rho_i(\Phi(\sum\limits_{i=1}^{n}\lambda_i(x_i-y_i)))=\rho_i(\sum\limits_{i=1}^{n}\lambda_i(\varphi(x_i)-\varphi(y_i)))\leq \\
 &\leq\max\limits_{1\leq i\leq n}|\lambda_i|\rho_{i_0}(\varphi(x_i)-\varphi(y_i))=\max\limits_{1\leq i\leq n}\rho_{i_0}(\varphi(x_i)-\varphi(y_i))<\varepsilon.
 	\end{align*}

 \end{proof}
It is known  that if a Tychonoff space $X$ is non-discrete, then $A(X)$ is not metrizable (see \cite[Theorem 7.1.20]{AT}).
This result inspired us to obtain the following.
 \begin{prop}
 	Let $(X,\mathcal U)$ be a non-discrete NA uniform space. Let
 	$\mathbb F$ be a complete NA valued field with a non-trivial valuation. Then $L_\mathbb{F}(X,{\mathcal U})$ is not metrizable.
 \end{prop}
 \begin{proof}
 	Assuming the contrary, there exists a decreasing sequence $\{U_n\}_{n\in \mathbb N}$ which forms a local base at  $\mathbf 0\in L_\mathbb{F}(X,{\mathcal U}).$ Since the valuation $|\cdot|$ is non-trivial, there exists $\lambda\in\mathbb F$ with $|\lambda|>1.$ In view of Theorem \ref{t:freeLCS} $(X,\mathcal U)$ is a uniform subspace of $L_\mathbb{F}(X,{\mathcal U}).$ By the continuity of the scalar multiplication it follows that there exists a sequence of entourages  $\varepsilon_n\in \mathcal U$ such that
 	$\lambda^n(x-y)\in U_n  \ \forall x,y\in \varepsilon_n.$ Since $\mathcal U$ is non-discrete and Hausdorff we can find a sequence $(x_n,y_n)\in \varepsilon_n$ such that $x_n\neq y_n \ \forall n\in\Bbb N$ and $\forall i<n \ \ x_n \notin \{x_i,y_i\}.$ Clearly, the sequence $u_n=\lambda^n(x_n-y_n) \in U_n$ converges to $\mathbf{0}.$ Let us show that this leads to a contradiction.
 	Since $(X,\mathcal U)$ is NA it is easy to define, by induction on $n,$ a sequence $\{f_n:n\in\Bbb N\}$ of uniformly continuous functions on $(X,\mathcal U)$ with values in $\mathbb F$ such that for every $n\geq 1$:
 	\begin{enumerate} \item $ \ |f_n(x)|\leq |\lambda|^{-n}  \   \forall x\in X;$
 	\item $f_n(x_k)=f_n(y_k)=f_n(y_n)=0_{\mathbb F} \ \ \forall k<n;$
 	\item $f_n(x_n)=\lambda^{-n}-\sum\limits_{k=1}^{n-1}(f_k(x_n)-f_k(y_n))$ if
 	$|\sum\limits_{k=1}^{n-1}(f_k(x_n)-f_k(y_n))|\leq |\lambda|^{-n}$ and
 	$f_n(x_n)=\lambda^{-n}$ otherwise.
 	\end{enumerate}
 	By $(3)$ and the strong triangle inequality we have $|f_n(x_n)+\sum\limits_{k=1}^{n-1}(f_k(x_n)-f_k(y_n))|\geq |\lambda|^{-n}.$  By $(1)$ for every $x\in X$ the  sequence of partial sums $\bigg\{\sum\limits_{k=1}^{n}f_k(x) \bigg \}_{n\in \Bbb N}$ is Cauchy.  Since the field $\mathbb F$ is complete,  the function $f=\sum\limits_{n=1}^{\infty}f_n$ is well defined. From $(1)$ it follows that $f$ is uniformly continuous, and thus it admits,   an extension to a linear continuous map $\widetilde{f}: L_\mathbb{F}(X,{\mathcal U})\to \mathbb F.$ For every $n\in \Bbb N$ we have
 	\begin{align*}
 |&\widetilde{f}(u_n)|=|\lambda|^{n}\cdot |\sum\limits_{k=1}^{\infty}(f_k(x_n)-f_k(y_n))|= \\ &=|\lambda|^{n}\cdot |f_n(x_n)+\sum\limits_{k=1}^{n-1}(f_k(x_n)-f_k(y_n))|\geq |\lambda|^{n}\cdot |\lambda|^{-n}=1.
 	\end{align*}
 	 It follows that the sequence $\{\widetilde{f}(u_n)\}$ does not converge to $\mathbf{0},$ contradicting the continuity of $\widetilde{f}.$
 \end{proof}
In contrast, note that the uniform free NA abelian
topological group  $A_{\scriptscriptstyle\mathcal{NA}}$ (Definition \ref{d:FreeGr}) is metrizable for every metrizable  NA uniform space $(X,{\mathcal U})$ (see \cite{MS} and also Remark  \ref{rem:spmet}).
 \subsection{Free abelian NA groups and  NA Tkachenko-Uspenskij theorem}

 Recall the following definition from  \cite{MS}. \begin{defi} \label{d:FreeGr} Let
 	$(X,{\mathcal U})$ be an NA uniform space. The \emph{uniform free NA abelian
 		topological group of $(X,{\mathcal U})$} is denoted by $A_{\scriptscriptstyle\mathcal{NA}}$ and  defined as follows:
 	$A_{\scriptscriptstyle\mathcal{NA}}$ is an  NA abelian topological group for
 	which
 	there exists a universal uniform map
 	$i: X \to A_{\scriptscriptstyle\mathcal{NA}}$
 	satisfying the following universal property. For every uniformly
 	continuous map $\varphi: (X,{\mathcal U}) \to G$ into an
 	abelian NA topological group $G$ there exists a unique
 	continuous homomorphism \textbf{$\Phi: A_{\scriptscriptstyle\mathcal{NA}} \to G $} for which
 	the following diagram commutes:
 	\begin{equation*} \label{equ:ufn}
 	\xymatrix { (X,{\mathcal U}) \ar[dr]_{\varphi} \ar[r]^{i} &
 		A_{\scriptscriptstyle\mathcal{NA}}
 		\ar[d]^{\Phi} \\
 		& G }
 	\end{equation*}
 \end{defi}

 Let $(X,{\mathcal U})$ be an NA uniform space and
 $Eq(\mathcal U)$ be the set of all equivalence relations from ${\mathcal
 	U}$.

 \begin{thm} \label{thm:desna} \cite[Theorem 4.14]{MS}
 	Let $(X,{\mathcal U})$ be NA and let
 	$\mathcal{B}\subseteq Eq(\mathcal U)$ be a  base of ${\mathcal U}$.
 	For every $\varepsilon\in \mathcal{B}$  denote by
 	$<\varepsilon>$ the subgroup of $A(X)$ algebraically generated by the set
 	$\{x-y\in A(X) : (x,y) \in \varepsilon\}.$
 	Then
 	$\{<\varepsilon>\}_{\varepsilon \in \mathcal{B} }$ is a local base at the zero element of 
 	$A_{\scriptscriptstyle\mathcal{NA}}(X,{\mathcal U}).$
 \end{thm}
\begin{remark} \label{rem:spmet}
 It is easy to see from the  description above that if $(X,d)$ is an ultra-metric space, then $A_{\scriptscriptstyle\mathcal{NA}}$ is metrizable. The following theorem provides a specific metrization which can be viewed as a Graev type  ultra-norm.
\end{remark}
 \begin{lemma}\label{lem:met}
 	 Let $(X,d)$ be an ultra-metric space treated as an ultra-metric subspace of  $(\overline{X},d)$ as in Lemma \ref{l:extend}. Then $A_{\scriptscriptstyle\mathcal{NA}}$ is metrizable by the Graev type ultra-norm defined as follows. For $u\in A(X)$ let
 	$$||u||:=\inf\bigg \{\max \limits_{1\leq i\leq n}d(x_i,y_i): u=\sum\limits_{i=1}^{n}(x_i-y_i), \ x_i,y_i\in \overline{X} \bigg \}.$$
 \end{lemma}
 \begin{proof}
 	Observe that for $\varepsilon<1$
 	we have $B_d(\mathbf{0},\varepsilon)=<\varepsilon>$,
 	where $B_d(\mathbf{0},\varepsilon)$ is the open $\varepsilon$-ball. 
 \end{proof}
 \begin{remark}
 	Suppose that $(X,\mathcal U)$ is an NA
 	uniform space generated by a collection of ultra-seminorms $\{d_i\}_{i\in I}.$ Then
 	using the idea of Lemma \ref{lem:met} one can show that the topology of $A_{\scriptscriptstyle\mathcal{NA}}$ is generated by the set of the corresponding Graev type ultra-norms
 	$\{||\cdot||_{d_i}\}_{i\in I}.$ So we have an analogy with Theorem \ref{t:freeLCS}.
 	At the same time we have one key difference. In the  description of $A_{\scriptscriptstyle\mathcal{NA}}$ it is enough to consider any  set of ultra-pseudometrics $\{d_i\}_{i\in I}$ which generate the uniformity $\mathcal U$ on $X.$
 \end{remark}

 By Tkachenko-Uspenskij theorem \cite{Tk,Us-free}, the free abelian topological group $A(X)$ is a  topological subgroup of $L(X)$ (here $\mathbb F=\Bbb R$).
 This can be derived (as in \cite{Us-free}) using the usual integer value property and descriptions of Graev's extension.
 Consider an NA valued field $\mathbb F$ of characteristic zero.
 It is clear that, algebraically,
 $A_{\scriptscriptstyle\mathcal{NA}}(X)$ is a natural subgroup of
 $L_{\mathbb F} (X)$ since $\Bbb Q$ is embedded in $\mathbb F$ as a subfield.
 So, it is  natural to ask for which NA valued fields $\mathbb F$ we have an analogue of Tkachenko-Uspenskij theorem. Theorem \ref{t: Tk-Usp} shows that
 this is true if and only if the valuation of $\mathbb F$ is trivial on $\Bbb Q$. First we give a particular example.

\begin{example} \label{r:Tk-Usp}
	Tkachenko-Uspenskij theorem is not true for the field $\mathbb F=\mathbb Q_{p}$ of $p$-adic numbers (with its standard valuation). Clearly, $\lim p^n =0_{\mathbb F}$ in ${\mathbb F}$.
	Now, let $x,y \in X$ be a pair of distinct points in an ultra-metric space $X$. By the continuity of the  operations $u_n:=p^n(x-y)$ converges to zero in the free locally convex space $L_{\mathbb F}(X)$. At the same time it is not true in the free NA abelian group $A_{\scriptscriptstyle\mathcal{NA}}(X),$ as it follows from the internal description of the topology of $A_{\scriptscriptstyle\mathcal{NA}}(X)$ (see Theorem \ref{thm:desna} or \cite{MS}).
\end{example}

 \begin{thm} \label{t: Tk-Usp}
 	Let $\mathbb F$ be an NA valued field and $(X,\mathcal U)$ be an
 	NA uniform space. Suppose also that $\operatorname{char}(\mathbb F)=0$ and consider 
  $A_{\scriptscriptstyle\mathcal{NA}} (X)$  as an algebraic subgroup of  $L_\mathbb F (X).$ The
 	following conditions are equivalent:
 	\begin{enumerate}
 	\item
 	 $A_{\scriptscriptstyle\mathcal{NA}}$ is a topological subgroup of $L_\mathbb{F}(X,{\mathcal U})$.
 	 \item
 	  The valuation of $\mathbb F$ is trivial on $\Bbb Q$.
 	  \end{enumerate}
 \end{thm}
 \begin{proof}
 	 (1) $\Rightarrow$ (2):
 	 If the valuation on $\Bbb Q$ is not trivial, then by Ostrowski's Theorem \ref{Ostrowski} this restricted valuation
 	 is equivalent to the 
 	  $p$-adic valuation. Now the proof is reduced to the concrete case of Example \ref{r:Tk-Usp}.
 	
 	 (2) $\Rightarrow$ (1):
 By Proposition \ref{p:emb} we know that  $L_\mathbb{Q}(X,{\mathcal U})$ is a topological subgroup of $L_\mathbb{F}(X,{\mathcal U})$. So it suffices to show that  $A_{\scriptscriptstyle\mathcal{NA}}$ is a topological subgroup of $L_\mathbb{Q}(X,{\mathcal U})$.	Let $\{d_i\}_{i\in I}$ be a family of ultra-pseudometrics generating the uniformity $\mathcal U.$ For every $i$ extend $d_i$ to $\overline X$
 as in Lemma \ref{l:extend}.
 Then consider the Kantorovich ultra-seminorm (Theorem \ref{t:AE}) $||\cdot||_{d_i}$ on $L_{\Bbb Q} (X).$ Since the restricted valuation $|\cdot|$ on $\Bbb Q$  is trivial, the topology of  $L_\mathbb{Q}(X,{\mathcal U})$ is generated by the family $\{||\cdot||_{d_i}\}_{i\in I}.$ It suffices to prove the following claim.
 	
 	\vskip 0.3cm
 	
 	\noindent	\textbf{Claim :} Let $(\overline{X},d)$ be an ultra-pseudometric
 	space,
 	 $||\cdot||^L$ be the corresponding Kantorovich ultra-seminorm on
 	$L_{\Bbb Q}(X)$ and $ ||\cdot||^A $ be the corresponding Graev type ultra-seminorm on $A_{\scriptscriptstyle\mathcal{NA}}$ (from Lemma \ref{lem:met}). Then
 	$||u||^L=||u||^A$ for every $u\in A(X)$.
 	\begin{proof}
 		Since $\Bbb Z$ is an additive subgroup of $\Bbb Q,$ it follows by Lemma \ref{l:GenDigital} that
 		\begin{align*}
 	&||u||^L=\inf\bigg \{\max \limits_{1\leq i\leq n}|\lambda_i|d(x_i,y_i): u=\sum\limits_{i=1}^{n}\lambda_i(x_i-y_i), \ x_i,y_i\in \overline{X}, \ \lambda_i\in \Bbb Z \bigg \} =\\ &=\inf\bigg \{\max \limits_{1\leq i\leq n}d(x_i,y_i): u=\sum\limits_{i=1}^{n}\lambda_i(x_i-y_i), \ x_i,y_i\in \overline{X}, \ \lambda_i\in \Bbb Z \bigg \}=\\ &=\inf\bigg \{\max \limits_{1\leq i\leq n}d(x_i,y_i): u=\sum\limits_{i=1}^{n}(x_i-y_i), \ x_i,y_i\in \overline{X} \bigg \}= ||u||^A.
 		\end{align*}
 	\end{proof}
\end{proof}
 \begin{example} \label{Levi-Civita}
 Theorem \ref{t: Tk-Usp} can be applied to the Levi-Civita field $\mathcal R$. Indeed, as it was noted in Example  \ref{Levi-Civita1}, $
  \mathcal R$ admits a natural dense valuation.
Its restriction on $\Bbb Q$ is trivial. We conclude, by
 Theorem \ref{t: Tk-Usp}, that  $A_{\scriptscriptstyle\mathcal{NA}}$ is a topological subgroup of
 $L_\mathcal{R}(X,{\mathcal U})$ for every NA uniform space $(X,{\mathcal U}).$
 \end{example}

 \section{Pointed version and the dual space}
 Using similar techniques to those mentioned in the previous sections, one can study the pointed version of NATP. However, its connection to the dual space is a unique feature which we present below.

 Let $(X,d,e) $ be a pointed ultra-pseudometric with a base point $e.$
 Let  $L_\mathbb F(X)$ be the free pointed 
 $\mathbb F$-vector space  on the pointed set $(X,e).$
 As before let $$L^{0}_\mathbb F(X):=\bigg \{ \sum\limits_{i=1}^{n}\lambda_i x_i\in L_\mathbb{F}(X)| \ \sum\limits_{i=1}^{n}\lambda_i=0_{\mathbb F} \bigg\}.$$

 \begin{defi}
 	The {\it Kantorovich ultra-seminorm} is the ultra-seminorm on $L^{0}_\mathbb F(X)$ given by the following formula. For $u\in L^{0}_\mathbb{F}(X)$ let
 	$$||u||:=\inf\bigg \{\max \limits_{1\leq i\leq n}|\lambda_i|d(x_i,y_i): u=\sum\limits_{i=1}^{n}\lambda_i(x_i-y_i), \ x_i,y_i\in X, \ \lambda_i\in \mathbb F \bigg \}.$$
 \end{defi}

It follows from the definition of the Kantorovich ultra-seminorm that $||x-y|| \leq d(x,y)$ for every $x,y \in X$.  As in the non-pointed case we can show
 that $||x-y|| = d(x,y)$ and $||\cdot||$ is an ultra-norm whenever $d$ is an ultra-metric.
 It is well known that the map $x \mapsto x-e$ defines an isometric embedding of a metric space $(X,d)$ into the classical Arens-Eells space. See, for example, \cite[Section 2.2]{Wea}. One may show that the same rule defines an isometric embedding of a pointed ultra-metric space $(X,d,e)$ into $(L^{0}_\mathbb F(X), ||\cdot||)$.
 For every  pointed Lipschitz function $f: X \to \mathbb F$ we have a canonically defined continuous functional
 $L^{0}_\mathbb F(X) \to \mathbb F.$
 Moreover, for a nontrivially valued NA field $\mathbb F,$  the dual NA Banach space of $L^{0}_\mathbb F(X)$ can be identified with the NA Banach space ${\rm Lip}_0$ of all  pointed Lipschitz functions  $f: X \to \mathbb F$. We omit the verification which essentially
 is very similar to the arguments of \cite[Theorem 2.2.2]{Wea}.  Note that the nontriviality of the valuation is important in order to ensure that every continuous functional $L^{0}_\mathbb F(X) \to \mathbb F$ is a Lipschitz function.
 See \cite[Prop. 3.1]{Schn}.

\section{Appendix}
\label{s:ap}

Let $(X,d)$ be a pseudometric space and $\Bbb C$ be the field of complex numbers.
As in the case of reals (Equation (\ref{secform})) define the Kantorovich seminorm
on $L^0_\Bbb C(X)$ as follows. For every $v \in L^0_\Bbb C(X) $  \begin{equation} \label{form}
	||v||=\inf\bigg\{\sum\limits_{i=1}^{l}|\rho_i|d(a_i,b_i):v=\sum\limits_{i=1}^{l}\rho_i(a_i-b_i), \ \rho_i\in \Bbb C, a_i,b_i\in X\bigg\}.
	\end{equation}
	
The following was mentioned in \cite{Flood, Wea, GaoPest}.

\begin{thm} \label{main}
	Support elements do not determine the Kantorovich norm for the field $\mathbb F:=\Bbb C$ of complex numbers.
\end{thm}	
The following example (which appears in \cite[p. 90]{Flood}, \cite[Ex. 1.5.7, p. 18]{Wea} without details) implies Theorem \ref{main}. That is, in general, the infimum  in (\ref{form}) cannot be achieved or even approximated by support elements. This was mentioned also in \cite{GaoPest}.
As we show below, one may say even more: in this example the infimum is attained outside the support.

\begin{example} \label{e:outside}
	Let $X=\{e,p,q,r\}$ and $d$ be a metric on $X$ defined as follows:
	$d(p,q)=d(p,r)=d(q,r)=1$ and 	$d(e,p)=d(e,q)=d(e,r)=\frac{1}{2}.$ Let $\lambda=1\cdot p+\mu q+\nu r\in L^0_\Bbb C(X),$ where $S=\{1,\mu,\nu\}$ denotes the set of the three complex
	cube roots of unity. We  show that the infimum  in the definition of $||\lambda||$ cannot be achieved or even approximated by support elements. We also show that the infimum is attained outside the support and that
	$||\lambda||=\frac{3}{2}.$
	
	\vskip 0.3cm \begin{enumerate}[(a)]
	
	\item	Since $1+\mu+\nu=0$ we have $\lambda= (p-e)+\mu(q-e)+\nu(r-e).$ It follows that
	$||\lambda||\leq d(p,e)+|\mu|d(q,e)+|\nu|d(r,e)=\frac{3}{2}.$ We will show that the minimal sum-cost which comes from presentations of $\lambda$  that include only support elements, is strictly larger than $\frac{3}{2}.$
	When dealing with support elements it suffices to consider presentations of $\lambda$ of the form $\lambda=c_{pq}(p-q)+ c_{pr}(p-r)+ c_{qr}(q-r)$ where
	$c_{pq},c_{pr},c_{qr}\in \Bbb C.$  Indeed, this follows from the reduction rules:
	\begin{enumerate} [(1)]\item Replace $m(x-y)$ with $-m(y-x).$
	\item Replace the terms $m(x-y), \ n(x-y)$ with $(m+n)(x-y).$ \end{enumerate}
	
	If $\lambda=c_{pq}(p-q)+ c_{pr}(p-r)+ c_{qr}(q-r)$ then $c_{pq}+c_{pr}=1,$  $-c_{pq}+c_{qr}=\mu,$ $   -c_{pr}-c_{qr}=\nu.$ So, the infimum is  $$\inf\{|c_{pq}|d(p,q)+|c_{pr}|d(p,r)+|c_{qr}|d(p,r):c_{pq}+c_{pr}=1,  -c_{pq}+c_{qr}=\mu, -c_{pr}-c_{qr}=\nu  \}.$$  Taking into account that 	$d(p,q)=d(p,r)=d(q,r)=1,$ we solve the system of linear equations and see that the latter expression is equal to  $\inf \limits_{t\in \Bbb C}(|\mu-t|+|0-t|+|-\nu-t|).$ Finding this infimum is a simple geometrical problem since $0,\mu,-\nu$ are three vertices of an equilateral triangle in the complex plane.  It follows that the infimum is equal to $\sqrt{3}.$ Clearly $\sqrt{3}>\frac{3}{2}$ as needed.

	\vskip 0.5cm \item We will show that the infimum is attained outside the support and that $||\lambda||=\frac{3}{2}.$ We already know that there exists a presentation of $\lambda$ for which the value of the sum-cost is $\frac{3}{2}.$ So $||\lambda||\leq\frac{3}{2}$  and it suffices to show that 	$||\lambda||\geq \frac{3}{2}.$
	This is done by showing that for every presentation of $\lambda$ of the form $\lambda=c_{ep}(e-p)+ c_{eq}(e-q)+ c_{er}(e-r)+c_{pq}(p-q)+ c_{pr}(p-r)+ c_{qr}(q-r),$ where 	$c_{ep},c_{eq},c_{er},c_{pq},c_{pr},c_{qr}\in \Bbb C,$ we have
	$$|c_{ep}|d(e,p)+	|c_{eq}|d(e,q)+|c_{er}|d(e,r)	+|c_{pq}|d(p,q)+|c_{pr}|d(p,r)+|c_{qr}|d(q,r)=$$$$=\frac{1}{2}(|c_{ep}|+|c_{eq}|+|c_{er}|)+|c_{pq}|+|c_{pr}|+|c_{qr}|\geq \frac{3}{2}.$$ We compare  the coefficients of $e,p,q,r$ in the normal presentation of $\lambda$ and in the "new" presentation and obtain   \begin{enumerate} [(1)]\item $c_{ep}+c_{eq}+c_{er}=0,$\item $-c_{ep}+c_{pq}+c_{pr}=1,$
	\item $-c_{eq}-c_{pq}+c_{qr}=\mu,$ \item  $-c_{er}-c_{pr}-c_{qr}=\nu.$ \end{enumerate}
	
	Now, using the triangle inequality and properties $(2)-(4),$ we obtain  $$\frac{1}{2}(|c_{ep}|+|c_{eq}|+|c_{er}|)+|c_{pq}|+|c_{pr}|+|c_{qr}|=\frac{1}{2}(|-c_{ep}|+|c_{pq}|+|c_{pr}|)+\frac{1}{2}(|-c_{eq}|+|-c_{pq}|+|c_{qr}|)+$$$$+\frac{1}{2}(|-c_{er}|+|-c_{pr}|+|-c_{qr}|)\geq \frac{1}{2}(|-c_{ep}+c_{pq}+c_{pr}|+|-c_{eq}-c_{pq}+c_{qr}|+|-c_{er}-c_{pr}-c_{qr}|)=$$$$=\frac{1}{2}(|1|+|\mu|+|\nu|)=\frac{3}{2}.$$ \end{enumerate}
\end{example}
We showed  that in the archimedean case the infimum can be attained outside the support.
In fact, as the following example shows, sometimes the infimum is not attained at all.
\begin{example}
	Let $X=\{p,q,r\}$ and $d$ be a metric on $X$ such that
	$d(p,q)=d(p,r)=d(q,r)=1.$ Let $\mathbb F=\Bbb Q(i)$ be the subfield of $\Bbb C,$ where   $\Bbb Q(i):=\{a+bi:\ a,b\in\Bbb Q\}.$ We will show that the infimum in the definition of $||u||$ is not attained in $\mathbb F$ for $u=(1-i)p+iq-r.$  It suffices to show that the infimum $$\inf\{|c_{pq}|d(p,q)+|c_{pr}|d(p,r)+|c_{qr}|d(q,r): c_{pq}+c_{pr}=1-i, \ -c_{pq}+c_{qr}=i,  \  -c_{pr}-c_{qr}=-1  \}= $$$$= \inf \limits_{t\in \mathbb F}(|t|+|t-i|+|t-1|)$$  is not attained. Since $\mathbb F=\Bbb Q(i)$ is a dense subfield of $\Bbb C$ it follows that the latter expression is equal to $\inf \limits_{t\in \Bbb C}(|t-i|+|1-t|+|t|).$ This infimum is attained at a unique point $p\in \Bbb C$ that is the Fermat-Torricelli point of the triangle in the complex plane with vertices
	$0,1,i.$ One can show that $p\notin \Bbb Q(i).$ By the uniqueness of $p$ it follows that the infimum in the definition of $||u||$ is not attained in $\mathbb F.$
	\end{example}
	 \section{Some possible developments and problems}

	 \begin{enumerate}
	 \item
	 One of the most attractive directions is the study of concrete applications of NATP
	 (non-archimedean transportation problem).
	 \item A natural perspective is to extend the discrete version of
	 NATP to a \emph{continuous} one (which in the classical case is based on measures).
	 \item It would be interesting to look for additional properties of
	 the free NA locally convex $\mathbb F$-space.
	 \end{enumerate}
	
	 \vskip 0.3cm
	 \noindent {\bf Acknowledgments.} We thank
	 R. Ben-Ari, A.M.  Brodsky, G. Luk\'{a}cs,  V. Pestov, L. Polev and
	 S.T. Rachev for valuable suggestions.
\bibliographystyle{amsalpha}

\end{document}